\documentclass[12pt,reqno]{amsart}
\usepackage[cp1251]{inputenc}
\usepackage[T2A]{fontenc}
\usepackage[english]{babel}
\usepackage{amsmath,amsfonts,amssymb}
\usepackage{geometry}

\numberwithin{equation}{section}

\newtheorem{theorem}{Theorem}[section]
\newtheorem{corollary}[theorem]{Corollary}
\newtheorem{lemma}[theorem]{Lemma}
\newtheorem{proposition}[theorem]{Proposition}

\theoremstyle{remark}
\newtheorem{remark}[theorem]{Remark}

\theoremstyle{definition}

\newtheorem*{main-definition}{Main Definition}

\begin{document}

\title[Elliptic problems with rough boundary data]{Elliptic problems with rough boundary data\\in generalized Sobolev spaces}


\author[A. Anop]{Anna Anop}

\address{Institute of Mathematics, National Academy of Sciences of Ukraine,
3 Tereshchenkivs'ka, Kyiv, 01004, Ukraine}

\email{anop@imath.kiev.ua}


\author[R. Denk]{Robert Denk}

\address{University of Konstanz, Department of Mathematics and Statistics, 78457  Konstanz,   Germany}

\email{robert.denk@uni-konstanz.de}


\author[A. Murach]{Aleksandr Murach}

\address{Institute of Mathematics, National Academy of Sciences of Ukraine,
3 Tereshchenkivs'ka, Kyiv, 01004, Ukraine}

\email{murach@imath.kiev.ua}




\thanks{The publication contains the results of studies conducted by the joint grant F81 of the National Research Fund of Ukraine and the German Research Society (DFG); competitive project F81/41686.}

\thanks{This work was supported by the Grant H2020-MSCA-RISE-2019, project number  873071 (SOMPATY: Spectral Optimization: From Mathematics to Physics and Advanced Technology).}

\thanks{The first author was supported by President of Ukraine's grant for competitive project F82/45932.}

\begin{abstract}
We investigate regular elliptic boundary-value problems in bounded domains and show the Fredholm property for the related operators in an extended scale formed by inner product Sobolev spaces (of arbitrary real orders) and corresponding interpolation Hilbert spaces. In particular, we can deal with boundary data with arbitrary low regularity. In addition, we show interpolation properties for the extended scale, embedding results, and global and local \textit{a priori} estimates for solutions to the problems under investigation. The results are applied to elliptic problems with homogeneous right-hand side and to elliptic problems with rough boundary data in Nikoskii spaces, which allows us to treat some cases of white noise on the boundary.
\end{abstract}

\maketitle

\section{Introduction}\label{sec1}

In this paper, we investigate elliptic boundary-value problems of the form
\begin{equation*}
A u =f\;\;\text{ in }\Omega,\quad
    B_j u = g_j \;\;\text{ on }\Gamma, \;\; j=1,\dots,q,
\end{equation*}
in classes of generalized Sobolev spaces. Here, $\Omega\subset\mathbb R^n$ is a bounded domain with boundary $\Gamma\in C^{\infty}$, $A$ is a linear partial differential operator (PDO) of order $2q$, and $B_j(x,D)$, $\nobreak{j=1,\dots,q}$, are linear boundary PDOs of order $m_j<2q$. We assume all coefficients to be infinitely smooth and the boundary-value problem $(A,B):=(A,B_1,\dots,B_q)$ to be regular elliptic. The aim of the present paper is the analysis of this problem in the so-called extended Sobolev scale of Hilbert distribution spaces. They are of the form $H^\alpha(\Omega)$, where $\alpha\in\textrm{OR}$ is an O-regularly varying function (see, e.g., \cite[Section~2.0.2]{BinghamGoldieTeugels89}). Note that the smoothness parameter $\alpha$ is a function, in contrast to the classical Sobolev spaces, where the smoothness is measured by some real number. The Hilbert spaces $H^\alpha(\Omega)$ are special cases of distribution spaces introduced by H\"ormander \cite{Hermander63, Hermander83} for a wide class of weight functions and based on the $L_p$-norm. In the situation considered here, the weight function is radially symmetric, and we restrict ourselves to the Hilbert space case of $p=2$. We remark that for $p=2$ the H\"ormander spaces coincide with the spaces introduced by Volevich and Paneah in \cite[Section~2]{VolevichPaneah65}.  The class $\{H^\alpha(\Omega):\alpha\in\textrm{OR}\}$ contains the classical Sobolev spaces $H^{r}(\Omega)$ with $r\in\mathbb{R}$ and can be seen as a finer scale of regularity, which allows for more precise embedding and trace theorems. On the other hand, the space $H^\alpha(\Omega)$ can be obtained from the classical Sobolev spaces by interpolation with a function parameter, see Section~\ref{sec5} below.

Recently, Mikhailets and Murach developed a general theory of solvability of elliptic boundary-value problems in a class of H\"ormander Hilbert spaces called the refined Sobolev scale (see \cite{MikhailetsMurach05UMJ5, MikhailetsMurach06UMJ3, MikhailetsMurach06UMJ11, MikhailetsMurach07UMJ5}, and the monograph \cite{MikhailetsMurach14}). The (larger) extended Sobolev scale was considered in \cite{AnopMurach14UMJ}. In these publications, the boundary data had sufficient regularity to guarantee the existence of boundary traces. More precisely, if $\alpha(t)\equiv\varphi(t)t^{2q}$, then the lower Matuszewska index of $\varphi$ was assumed to be larger than $-1/2$ (see Section~\ref{sec3} and Proposition~\ref{prop1} below for details). Motivated by applications with rough boundary data, in this paper we consider the situation where this condition on the
Matuszewska index does not hold. Even for Sobolev spaces, the case of rough boundary data is quite sophisticated. One approach is the modification of the Sobolev spaces with low regularity as developed by Roitberg \cite{Roitberg64,Roitberg96,Roitberg99}. Another way to treat this problem is to include the norm of $Au$ in the norm of the Sobolev space, see Lions and Magenes \cite[Chapter~2, Section~6]{LionsMagenes72}. In  connection with negative order boundary spaces, we also refer to \cite{GesztesyMitrea08} for recent results on weak and very weak traces and to \cite[Chapter~5]{BehrndtHassideSnoo20} for the theory of boundary triplets.

This paper has the following structure: Section~\ref{sec2} contains the precise formulation of the boundary-value problem $(A,B)$; in Section~\ref{sec3} we introduce the extended Sobolev scales over $\mathbb R^n$, $G$, and $\Gamma$. The main results are formulated in Section~\ref{sec4}. We show here that $(A,B)$ induces a Fredholm operator in the extended Sobolev scale (Theorem~\ref{th1}). We obtain global and local (up to the boundary) elliptic regularity in the extended scale (see Theorems~\ref{th4.6} and~\ref{th4.7}, resp.) and elliptic \textit{a priori} estimates (see Theorem~\ref{th4.12} for the global and Theorem~\ref{th4.13} for the local version). Theorems \ref{th4.7} and \ref{th4.13} are new even in the case of Sobolev spaces. In Section~\ref{sec5}, we discuss interpolation properties of the extended Sobolev scale, which will also be used in the proof of the main results in Section~\ref{sec6}. In Section~\ref{sec7}, we study semi-homogeneous boundary value problems, namely the case of $f=0$. Defining the space $H^\alpha_A(\Omega):=\{ u\in H^\alpha(\Omega): Au =0\}$, we obtain, e.g.,  conditions for uniform convergence of sequences of solutions to the homogeneous elliptic equation (Theorems \ref{th7.5} and \ref{th7.6}) and interpolation properties for $H^\alpha_A(\Omega)$ (Theorems \ref{th7.8} and \ref{th7.9}). Finally, in Section~\ref{sec8} we apply the results to elliptic boundary-value problems whose boundary data belong to some Nikolskii space $B^{s}_{2,\infty}(\Gamma)$. Based on an embedding result (Proposition~\ref{8.1}), we show that the solution belongs pathwise to the space $H^\alpha(\Omega)$ under some condition on $\alpha$. The investigation of such boundary-value problems is motivated by recent results on boundary noise (see, e.g., \cite{SchnaubeltVeraar11}) and on the Besov smoothness of white noise \cite{FageotFallahUnser17,Veraar11}.

\section{Statement of the problem}\label{sec2}
Let $\Omega\subset\mathbb{R}^n$, where $n\geq2$, be a bounded domain with an infinitely smooth boundary $\Gamma$. We consider the following boundary value problem:
\begin{align}
Au& =f\quad\mbox{in}\;\Omega,\label{f1}\\
B_{j}u& =g_{j}\quad\mbox{on}\;\Gamma,\quad j=1,...,q.\label{f2}
\end{align}
Here,
$$
A:=A(x,D):=\sum_{|\mu|\leq 2q}a_{\mu}(x)D^{\mu}\
$$
is a linear PDO on $\overline{\Omega}:=\Omega\cup\Gamma$ of even order $2q\geq2$, and each
$$
B_{j}:=B_{j}(x,D)=\sum_{|\mu|\leq m_{j}}b_{j,\mu}(x)D^{\mu}\
$$
is a linear boundary PDO on $\Gamma$ of order $m_{j}\leq2q-1$. All the coefficients $a_{\mu}$ and $b_{j,\mu}$ of these PDOs belong to the complex spaces $C^{\infty}(\overline{\Omega})$ and $C^{\infty}(\Gamma)$, resp. Let $B:=(B_{1},\ldots,B_{q})$ and $g:=(g_{1},\ldots,g_{q})$.

We use the following standard notation:
$\mu:=(\mu_{1},\ldots,\mu_{n})$ is a multi-index with nonnegative integer components, $|\mu|:=\mu_{1}+\cdots+\mu_{n}$,
$D^{\mu}:=D_{1}^{\mu_{1}}\cdots D_{n}^{\mu_{n}}$, $D_{k}:=i\partial/\partial x_{k}$, $k=1,...,n$, where $i$ is imaginary unit and $x=(x_1,\ldots,x_n)$ is an arbitrary point in $\mathbb{R}^{n}$.

We suppose throughout the paper that the boundary value problem \eqref{f1}, \eqref{f2} is regular elliptic in $\Omega$. This means that the PDO $A$ is properly elliptic on $\overline{\Omega}$ and that the system $B$ of boundary PDOs is normal and satisfies the Lopatinskii condition with respect to $A$ on $\Gamma$ (see, e.g., the survey \cite[Section~1.2]{Agranovich97}). Recall that, since the system $B$ is normal, the orders $m_{j}$ of $B_{j}$ are all different.

We investigate properties of the extension (by continuity) of the mapping \begin{equation}\label{mapping}
u\mapsto(Au,Bu)=(Au,B_{1}u,\ldots,B_{q}u),\quad\mbox{where}\quad
u\in C^{\infty}(\overline{\Omega}),
\end{equation}
on appropriate pairs of Hilbert distribution spaces. To describe the range of this extension, we need the following Green's formula:
\begin{equation*}
(Au,v)_{\Omega} + \sum^{q}_{j=1}(B_{j}u,C^{+}_{j}v)_{\Gamma} =
(u,A^{+}v)_{\Omega} + \sum_{j=1}^{q}(C_{j}u,B^{+}_{j}v)_{\Gamma}
\end{equation*}
for arbitrary $u,v\in C^{\infty}(\overline{\Omega})$. Here,
$$
A^{+}v(x):=\sum_{|\mu|\leq2q}D^{\mu}(\overline{a_{\mu}(x)}\,v(x))
$$
is the linear PDO which is formally adjoint to $A$, and
$\{B^+_j\}$, $\{C_j\}$, $\{C^+_j\}$ are some normal sets of linear boundary PDOs with coefficients from $C^\infty(\Gamma)$. The orders of these PDOs satisfy the condition
$$
\mathrm{ord}\,B_j+\mathrm{ord}\,C^+_j=\mathrm{ord}\,C_j+\mathrm{ord}\,B^+_j
=2q-1.
$$
In Green's formula and below, $(\cdot,\cdot)_\Omega$ and $(\cdot,\cdot)_\Gamma$ denote the inner products in the complex Hilbert spaces $L_2(\Omega)$ and $L_2(\Gamma)$ of all
functions that are square integrable over $\Omega$ and $\Gamma$, respectively (relative to the Lebesgue measure, of course), and also denote extensions by continuity of these inner products.

The boundary value problem
\begin{align}\label{f3}
A^{+}v&=w,\quad\mbox{in}\;\Omega,\\
B^{+}_{j}v&=h_{j},\quad\mbox{on}\;\Gamma,\quad j=1,\ldots,q,
\label{f4}
\end{align}
is called formally adjoint to the problem \eqref{f1}, \eqref{f2}
with respect to the given Green formula. The latter problem is regular  elliptic if and only if the formally adjoint problem \eqref{f3}, \eqref{f4} is regular elliptic \cite[Chapter~2, Section~2.5]{LionsMagenes72}.

Denote
\begin{gather*}
N:=\bigl\{u\in C^{\infty}(\overline{\Omega}):
\,Au=0\;\,\mbox{in}\;\,\Omega,\;\,
Bu=0\;\,\mbox{on}\;\,\Gamma\bigr\},\\
N^{+}:=\bigl\{v\in C^{\infty}(\overline{\Omega}):
\,A^{+}v=0\;\,\mbox{in}\;\,\Omega,\;\,
B^{+}v=0\;\,\mbox{on}\;\,\Gamma\bigr\},
\end{gather*}
with $B^{+}:=(B^{+}_{1},\ldots, B^{+}_{q})$. Since both problems \eqref{f1}, \eqref{f2} and \eqref{f3}, \eqref{f4} are regular elliptic, both spaces $N$ and $N^+$ are finite-dimensional \cite[Chapter~2, Section~2.5]{LionsMagenes72}. Besides, the space $N^{+}$ is independent of any choice of the collection $B^{+}$ of boundary differential expressions that satisfy Green's formula.

\section{Generalized Sobolev spaces}\label{sec3}

We investigate the boundary value problem \eqref{f1}, \eqref{f2} in certain Hilbert distribution spaces that are generalizations of inner product Sobolev spaces (of an arbitrary real order) to the case where a general enough function parameter is used as an order of the space. Such spaces were introduced and investigated by Malgrange \cite{Malgrange57}, H\"ormander \cite[Sec. 2.2]{Hermander63}, and Volevich and Paneah \cite[Section~2]{VolevichPaneah65}.

This function parameter ranges over a certain class OR of O-regularly varying functions.
By definition, OR is the class of all Borel measurable functions
$\alpha:[1,\infty)\rightarrow(0,\infty)$ such that
\begin{equation}\label{f3.1}
c^{-1}\leq\frac{\alpha(\lambda t)}{\alpha(t)}\leq c\quad\mbox{for arbitrary}\quad t\geq1\quad\mbox{and}\quad\lambda\in[1,b]
\end{equation}
with some numbers $b>1$ and $c\geq1$ that are independent of both $t$ and $\lambda$ (but may depend on $\alpha$). Such functions are called O-regularly varying at infinity in the sense of Avakumovi\'c \cite{Avakumovic36} and are well investigated \cite{BinghamGoldieTeugels89, BuldyginIndlekoferKlesovSteinebach18, Seneta76}.

The class OR admits the simple description
\begin{equation*}
\alpha\in\mathrm{OR}\quad\Longleftrightarrow\quad\alpha(t)=
\exp\Biggl(\beta(t)+\int\limits_{1}^{t}\frac{\gamma(\tau)}{\tau}\;
d\tau\Biggr), \;\;t\geq1,
\end{equation*}
where the real-valued functions $\beta$ and $\gamma$ are Borel measurable and bounded on $[1,\infty)$. Condition \eqref{f3.1} is equivalent to the following: there exist real numbers $r_{0}\leq r_{1}$ and positive numbers $c_{0}$ and $c_{1}$ such that
\begin{equation}\label{f3.2}
c_{0}\lambda^{r_{0}}\leq\frac{\alpha(\lambda t)}{\alpha(t)}\leq
c_{1}\lambda^{r_{1}}\quad\mbox{for all}\quad
t\geq1\quad\mbox{and}\quad\lambda\geq1.
\end{equation}
For every function $\alpha\in\mathrm{OR}$, we define the lower and the upper Matuszewska indices \cite{Matuszewska64} by the formulas
\begin{gather}\label{f3.2sup}
\sigma_{0}(\alpha):=\sup\{r_{0}\in\mathbb{R}:\,\mbox{the left-hand inequality in \eqref{f3.2} holds}\},\\
\sigma_{1}(\alpha):=\inf\{r_{1}\in\mathbb{R}:\,\mbox{the right-hand inequality in \eqref{f3.2} holds}\},\label{f3.2inf}
\end{gather}
with $-\infty<\sigma_{0}(\alpha)\leq\sigma_{1}(\alpha)<\infty$ (see also \cite[Theorem~2.2.2]{BinghamGoldieTeugels89}).

A standard example of functions from $\mathrm{OR}$ is given by a continuous function $\alpha:[1,\infty)\rightarrow(0,\infty)$ such that
\begin{equation*}
\alpha(t):=t^{r}(\log t)^{k_{1}}(\log\log
t)^{k_{2}}\ldots(\underbrace{\log\ldots\log}_{j\;\mathrm{times}} t)^{k_{j}}\quad\mbox{for}\quad t\gg1.
\end{equation*}
Here, we arbitrarily choose an integer $j\geq1$ and real numbers $r,k_{1},\ldots,k_{j}$. This function has equal Matuszewska indices $\sigma_{0}(\alpha)=\sigma_{1}(\alpha)=r$.

Generally, the class OR contains an arbitrary Borel measurable function $\alpha:[1,\infty)\rightarrow(0,\infty)$ such that both functions $\alpha$ and $1/\alpha$ are bounded on every bounded subset of $[1,\infty)$ and that the function $\alpha$ is regularly varying at infinity in the sense of Karamata, i.e. there exists a real number $r$ such that
\begin{equation*}
\lim_{t\rightarrow\infty}\;\frac{\alpha(\lambda\,t)}{\alpha(t)}=
\lambda^{r}\quad\mbox{for every}\;\lambda>0.
\end{equation*}
In this case $\sigma_{0}(\alpha)=\sigma_{1}(\alpha)=r$, and $r$ is called the order of $\alpha$. If $r=0$, the function $\alpha$ is called slowly varying at infinity.

A simple example of a function $\alpha\in\mathrm{OR}$ with the different Matuszewska indices is given by the formula
\begin{equation*}
\alpha(t):=\left\{
\begin{array}{ll}
t^{\theta+\delta\sin((\log\log t)^{\lambda})}\; &\hbox{if}\;\;t>e,\\
t^{\theta}\; &\hbox{if}\;\;1\leq t\leq e.
\end{array}\right.
\end{equation*}
Here, we arbitrarily choose numbers $\theta\in\mathbb{R}$, $\delta>0$, and $\lambda\in(0,1]$. Then $\sigma_{0}(\alpha)=\theta-\delta$ and $\sigma_{1}(\alpha)=\theta+\delta$ whenever $0<\lambda<1$, but $\sigma_{0}(\alpha)=\theta-\sqrt{2}\delta$ and $\sigma_{1}(\alpha)=\theta+\sqrt{2}\delta$ if $\lambda=1$. If $\lambda>1$, then $\alpha\notin\mathrm{OR}$.

Let $\alpha\in\mathrm{OR}$, and introduce the generalized Sobolev spaces  $H^{\alpha}$ over $\mathbb{R}^{n}$, with $n\geq1$, and then over $\Omega$ and~$\Gamma$. We consider complex-valued functions and distributions and therefore use complex linear spaces. It is useful for us to interpret distributions as antilinear functionals on a relevant space of test functions.

By definition, the linear space $H^{\alpha}(\mathbb{R}^{n})$ consists of all distributions $w\in\mathcal{S}'(\mathbb{R}^{n})$ such that their Fourier transform $\widehat{w}:=\mathcal{F}w$ is locally Lebesgue integrable over $\mathbb{R}^{n}$ and satisfies the condition
$$
\int\limits_{\mathbb{R}^{n}}\alpha^2(\langle\xi\rangle)\,
|\widehat{w}(\xi)|^2\,d\xi<\infty.
$$
As usual, $\mathcal{S}'(\mathbb{R}^{n})$ is the linear topological
space of tempered distributions in $\mathbb{R}^{n}$, and
$\langle\xi\rangle:=(1+|\xi|^{2})^{1/2}$ whenever
$\xi\in\mathbb{R}^{n}$. The space $H^{\alpha}(\mathbb{R}^{n})$ is endowed with the inner product
$$
(w_1,w_2)_{\alpha,\mathbb{R}^{n}}:=
\int_{\mathbb{R}^{n}}\alpha^2(\langle\xi\rangle)\,
\widehat{w_1}(\xi)\,\overline{\widehat{w_2}(\xi)}\,d\xi,
$$
and the corresponding norm $\|w\|_{\alpha,\mathbb{R}^{n}}:=
(w,w)_{\alpha,\mathbb{R}^{n}}^{1/2}$. We call $\alpha$ the order or regularity index of $H^{\alpha}(\mathbb{R}^{n})$ (and its analogs for $\Omega$ and $\Gamma$).

The space $H^{\alpha}(\mathbb{R}^{n})$ is an isotropic Hilbert case of the spaces $B_{p,k}$ introduced and systematically investigated by H\"ormander \cite[Section~2.2]{Hermander63} (see also \cite[Section~10.1]{Hermander83}). Namely,
$H^{\alpha}(\mathbb{R}^{n})=B_{p,k}$ provided that $p=2$ and
$k(\xi)=\alpha(\langle\xi\rangle)$ for all $\xi\in\mathbb{R}^{n}$.

If $\alpha(t)\equiv t^{r}$, then
$H^{\alpha}(\mathbb{R}^{n})=:H^{r}(\mathbb{R}^{n})$ is the inner product Sobolev space of order $r\in\mathbb{R}$. Generally,
\begin{equation}\label{f3.3}
r_{0}<\sigma_{0}(\alpha)\leq\sigma_{1}(\alpha)<r_{1}\;\;\Rightarrow\;\;
H^{r_1}(\mathbb{R}^{n})\hookrightarrow H^{\alpha}(\mathbb{R}^{n})\hookrightarrow
H^{r_0}(\mathbb{R}^{n}),
\end{equation}
both embeddings being continuous and dense. This is a consequence of the property \eqref{f3.2} considered for $t=1$.

A relation between $H^\alpha(\mathbb{R}^{n})$ and the space of $p$ times continuously differentiable functions reveals H\"ormander's embedding theorem \cite[Theorem~2.2.7]{Hermander63}, which is formulated in the $\alpha\in\mathrm{OR}$ case as follows \cite[Lemma~2]{ZinchenkoMurach12UMJ11}:
\begin{equation}\label{Hermander-embedding}
\int\limits_1^{\infty} t^{2p+n-1}\alpha^{-2}(t)\,dt<\infty\;\;\Longleftrightarrow\;\;
\{w\in H^\alpha(\mathbb{R}^{n}):\mathrm{supp}\,w\subset U\}\subset C^p(\mathbb{R}^{n});
\end{equation}
here, $0\leq p\in\mathbb{Z}$, and $U$ is an open nonempty subset of $\mathbb{R}^{n}$ (the case of $U=\mathbb{R}^{n}$ is possible).

Remark that we use the same designation $H^{\alpha}$ both in the case where $\alpha$ is a function and in the case where $\alpha$ is a number. This will not lead to ambiguity because we will always specify what $\alpha$ means, a function or number. This remark also concerns designations of spaces induced by $H^{\alpha}(\mathbb{R}^{n})$ and, of course, the notation of the norm and inner product in the corresponding spaces.

Following \cite{MikhailetsMurach13UMJ3}, we call the class
$\{H^{\alpha}(\mathbb{R}^{n}):\alpha\in\mathrm{OR}\bigr\}$
the extended Sobolev scale over $\mathbb{R}^{n}$. Its analogs for  $\Omega$ and $\Gamma$ are introduced in the standard way (see
\cite[Section~2]{MikhailetsMurach15ResMath1} and \cite[Section~2.4.2]{MikhailetsMurach14}, resp.). Let us give the necessary definitions.

By definition,
\begin{gather}\notag
H^{\alpha}(\Omega):=\bigl\{w\!\upharpoonright\!\Omega:
w\in H^{\alpha}(\mathbb{R}^{n})\bigr\},\\
\|u\|_{\alpha,\Omega}:=
\inf\bigl\{\,\|w\|_{\alpha,\mathbb{R}^{n}}:
w\in H^{\alpha}(\mathbb{R}^{n}),\;w=u\;\,\mbox{in}\;\,\Omega\bigr\}, \label{4f9}
\end{gather}
with $u\in H^{\alpha}(\Omega)$. The linear space $H^{\alpha}(\Omega)$ is Hilbert
and separable with respect to the norm \eqref{4f9} because $H^{\alpha}(\Omega)$ is
the factor space of the separable Hilbert space $H^\alpha(\mathbb{R}^{n})$ by its subspace
\begin{equation}\label{f3.5}
\bigl\{w\in H^{\alpha}(\mathbb{R}^{n}):\,
\mathrm{supp}\,w\subseteq\mathbb{R}^{n}\setminus\Omega\bigr\}.
\end{equation}
The norm \eqref{4f9} is induced by the inner product
\begin{equation*}
(u_1,u_2)_{\alpha,\Omega}:=
(w_1-\Pi w_1,w_2-\Pi w_2)_{\alpha,\mathbb{R}^{n}}.
\end{equation*}
Here, $u_j\in H^\alpha(\Omega)$, $w_j\in
H^{\alpha}(\mathbb{R}^{n})$, and $u_j=w_j$ in $\Omega$ for each $j\in\{1,2\}$, whereas $\Pi$ is the orthoprojector of $H^{\alpha}(\mathbb{R}^{n})$ onto \eqref{f3.5}.

The space $H^\alpha(\Omega)$ is continuously embedded in the linear topological space $\mathcal{D}'(\Omega)$ of all distributions in $\Omega$, and the set $C^\infty(\overline{\Omega})$ is dense in $H^\alpha(\Omega)$. Note that $H^\alpha(\Omega)$ is an isotropic case of Hilbert spaces introduced and investigated by Volevich and Paneah \cite[Section~3]{VolevichPaneah65}.

The linear space $H^{\alpha}(\Gamma)$ consists of all distributions on $\Gamma$ that yield elements of $H^{\alpha}(\mathbb{R}^{n-1})$ in local coordinates on $\Gamma$. Let us give a detailed definition. The boundary $\Gamma$ of $\Omega$ is an infinitely smooth closed manifold of dimension $n-1$, with the $C^{\infty}$-structure on $\Gamma$ being induced by $\mathbb{R}^{n}$. From this structure, we choose a finite collection of local charts $\pi_{j}:\mathbb{R}^{n-1}\leftrightarrow\Gamma_{j}$, $j=1,\ldots,\varkappa$, where the open sets $\Gamma_{j}$ form a covering of $\Gamma$. We also choose functions
$\chi_{j}\in C^{\infty}(\Gamma)$, $j=1,\ldots,\varkappa$, that satisfy the condition $\mathrm{supp}\,\chi_{j}\subset\Gamma_{j}$ and that form a partition of unity on $\Gamma$. Then
\begin{equation*}
H^{\alpha}(\Gamma):=\bigl\{h\in\mathcal{D}'(\Gamma):\,
(\chi_{j}h)\circ\pi_{j}\in H^{\alpha}(\mathbb{R}^{n-1})\;\;\mbox{for
every}\;\;j\in\{1,\ldots,\varkappa\}\bigr\}.
\end{equation*}
Here, as usual, $\mathcal{D}'(\Gamma)$ denotes the linear topological space of all distributions on~$\Gamma$, and $(\chi_{j}h)\circ\pi_{j}$ stands for the representation of the distribution $\chi_{j}h$ in the local chart $\pi_{j}$. The space $H^{\alpha}(\Gamma)$ is endowed with the
inner product
$$
(h_{1},h_{2})_{\alpha,\Gamma}:=
\sum_{j=1}^{\varkappa}\,((\chi_{j}h_{1})\circ\pi_{j},
(\chi_{j}\,h_{2})\circ\pi_{j})_{\alpha,\mathbb{R}^{n-1}}
$$
and the corresponding norm $\|h\|_{\alpha,\Gamma}:=(h,h)_{\alpha,\Gamma}^{1/2}$. The space
$H^\alpha (\Gamma)$ is Hilbert and separable and does not depend (up to equivalence of norms) on our choice of local charts and partition of unity on~$\Gamma$ \cite[Theorem~2.21]{MikhailetsMurach14}. This space is continuously embedded in $\mathcal{D}'(\Gamma)$, and  the set $C^{\infty}(\Gamma)$ is dense in $H^{\alpha}(\Gamma)$.

The above-defined function spaces form the extended Sobolev scales $\{H^{\alpha}(\Omega):\alpha\in\mathrm{OR}\}$ and $\{H^{\alpha}(\Gamma):\alpha\in\mathrm{OR}\}$ over $\Omega$ and $\Gamma$ respectively. They contain the scales of inner product
Sobolev spaces; namely, if $\alpha(t)\equiv t^r$ for certain $r\in\mathbb{R}$, then $H^\alpha(\Omega)=:H^{r}(\Omega)$ and $H^\alpha(\Gamma)=:H^{r}(\Gamma)$ are
the Sobolev spaces of order~$r$. Property \eqref{f3.3} remains true provided that we replace $\mathbb{R}^{n}$ with $\Omega$ or $\Gamma$,
the embeddings being compact. As we have noted, the norm in $H^{r}(G)$ is denoted by $\|\cdot\|_{r,G}$, with $G\in\{\mathbb{R}^{n},\Omega,\Gamma\}$.

The extended Sobolev scales have important interpolation properties: they are obtained by the interpolation with a function parameter between inner product Sobolev spaces, are closed with respect to the interpolation with a function parameter between Hilbert spaces, and consist (up to equivalence of norms) of all Hilbert spaces that are interpolation ones between inner product Sobolev spaces. We will discuss these properties in Section~\ref{sec5}. The first of them plays a key role in applications of these scale to elliptic operators and elliptic problems.

\section{Main results}\label{sec4}

Dealing with the problem \eqref{f1}, \eqref{f2}, we will use the generalized Sobolev space $H^{\alpha}(G)$, with $G\in\{\Omega,\Gamma\}$, whose order is a function parameter of the form  $\alpha(t)\equiv\varphi(t)t^{s}$ where $\varphi\in\mathrm{OR}$ and $s\in\mathbb{R}$. In order not to indicate the argument $t$ of the function parameter, we resort to the function $\varrho(t):=t$ of $t\geq1$. Then $\alpha$ can be written as $\varphi\varrho^{s}$ not using $t$. Note if $\varphi\in\mathrm{OR}$ and $s\in\mathbb{R}$, then $\varphi\rho^s\in\mathrm{OR}$ and $\sigma_j(\varphi\varrho^s)=\sigma_j(\varphi)+s$ for each $j\in\{0,1\}$.

It is well known that the elliptic problem \eqref{f1}, \eqref{f2} is Fredholm on appropriate pairs of Sobolev spaces of sufficiently large orders and that its index does not depend on these orders (see, e.g., \cite[Section 2.4~a]{Agranovich97} or \cite[Chapter~2, Section~5.4]{LionsMagenes72}). This result was extended to generalized Sobolev spaces in \cite[Theorem~1]{AnopMurach14UMJ} as follows:

\begin{proposition}\label{prop1}
Let $\varphi\in\mathrm{OR}$, and suppose that $\sigma_0(\varphi)>-1/2$. Then the mapping \eqref{mapping} extends uniquely (by continuity) to a bounded linear operator
\begin{equation}\label{f9}
(A,B):H^{\varphi\varrho^{2q}}(\Omega)\rightarrow
H^{\varphi}(\Omega)\oplus
\bigoplus_{j=1}^{q}H^{\varphi\varrho^{2q-m_j-1/2}}(\Gamma)
=:\mathcal{H}^\varphi(\Omega,\Gamma).
\end{equation}
This operator is Fredholm. Its kernel coincides with $N$, and its range  consists of all vectors
$(f,g)\in\mathcal{H}^\varphi(\Omega,\Gamma)$ such that
\begin{equation}\label{f10}
(f,v)_{\Omega}+\sum_{j=1}^{q}\,(g_{j},C^{+}_{j}v)_{\Gamma}=0\quad \mbox{for each}\quad v\in N^{+}.
\end{equation}
The index of the operator \eqref{f9} equals $\dim N-\dim N^{+}$ and does not depend on~$\varphi$.
\end{proposition}

Recall that the bounded linear operator $T:\nobreak E_{1}\rightarrow E_{2}$ between Banach spaces $E_{1}$ and $E_{2}$ is called Fredholm if its kernel $\ker T$ and cokernel $E_{2}/T(E_{1})$ are finite-dimensional. If $T$ is Fredholm, then its range $T(E_{1})$ is closed in $E_{2}$ (see, e.g.,  \cite[Лемма~19.1.1]{Hermander85}) and the index $\mathrm{ind}\,T:=\dim\ker T-\dim(E_{2}/T(E_{1}))$ is finite.

As to formula \eqref{f10}, recall that the inner product in $L_{2}(\Omega)$
extends by continuity to a sesquilinear form $(f,v)_{\Omega}$ of arbitrary arguments $f\in H^{-1/2+}(\Omega)$ and $v\in H^{1/2}(\Omega)$ (see, e.g., \cite[Theorem~4.8.2(b)]{Triebel95}). Here and below,
\begin{equation*}
H^{r+}(\Omega):=\bigcup_{\ell>r}H^{\ell}(\Omega)=
\bigcup_{\substack{\varphi\in\mathrm{OR},\\\sigma_0(\varphi)>r}}
H^{\varphi}(\Omega)\quad\mbox{for every}\quad r\in\mathbb{R}.
\end{equation*}
Thus, the first summand in \eqref{f10} is well defined. The next summands are also well defined being equal to the value of the distribution $g_{j}\in\mathcal{D}'(\Gamma)$ at the test function $C^{+}_{j}v\in C^{\infty}(\Gamma)$.

Proposition~\ref{prop1} is not true in the case of $\sigma_0(\varphi)\leq-1/2$. This is stipulated by the fact that the mapping $u\mapsto B_{j}u$, where $u\in C^{\infty}(\overline{\Omega})$, can not be extended to the continuous linear operator $B_{j}:H^{(s+2q)}(\Omega)\rightarrow\mathcal{D}'(\Gamma)$ if $s+2q\leq m_{j}+1/2$ (see \cite[Chapter~1, Теорема~9.5]{LionsMagenes72}). Therefore to obtain a version of Proposition~\ref{prop1} in this case, we have to take a narrower space than $H^{\varphi\varrho^{2q}}(\Omega)$ as the domain of the operator $(A,B)$. We will show that it is possible for this purpose to take the space of all distributions $u\in H^{\varphi\varrho^{2q}}(\Omega)$ such that $Au\in H^{\eta}(\Omega)$ for certain $\eta\in\mathrm{OR}$ subject to $\sigma_0(\eta)>-1/2$.

Let us previously consider this space for arbitrary function parameters $\alpha:=\varphi\varrho^{2q}$ and $\eta$ from $\mathrm{OR}$.  We put
\begin{equation}\label{f11}
H^{\alpha}_{A,\eta}(\Omega):=
\bigl\{u\in H^{\alpha}(\Omega):\,Au\in H^\eta(\Omega)\bigr\},
\end{equation}
with $Au$ being understood in the sense of the theory of distributions.  The linear space \eqref{f11} is endowed with the graph inner product
\begin{equation*}
(u_1,u_2)_{\alpha,A,\eta}:=
(u_1,u_2)_{\alpha,\Omega}+(Au_1,Au_2)_{\eta,\Omega}
\end{equation*}
and the corresponding norm $\|u\|_{\alpha,A,\eta}:=(u,u)_{\alpha,A,\eta}^{1/2}$.

The space $H^{\alpha}_{A,\eta}(\Omega)$ is complete, i.e. Hilbert. Indeed if $(u_{k})$ is a Cauchy sequence in this space,  there exist limits $u:=\lim u_{k}$ in $H^{\alpha}(\Omega)$ and  $f:=\lim Au_{k}$ in $H^{\eta}(\Omega)$. Since the PDO $A$ is continuous in $\mathcal{D}'(\Omega)$, the first limit implies that $Au=\lim Au_{k}$ in $\mathcal{D}'(\Omega)$. Hence, $Au=f\in H^{\eta}(\Omega)$. Therefore, $u\in H^{\alpha}_{A,\eta}(\Omega)$ and $\lim u_{k}=u$ in the space $H^{\alpha}_{A,\eta}(\Omega)$, i.e. this space is complete.

If $\alpha(t)\equiv t^{r}$ and $\eta(t)\equiv t^{\lambda}$ for some $r,\lambda\in\mathbb{R}$ (the Sobolev case), the space $H^{r}_{A,\lambda}(\Omega):=H^{\alpha}_{A,\eta}(\Omega)$ is investigated in \cite{KasirenkoMikhailetsMurach19}. This space is used in the theory of elliptic problems in negative Sobolev spaces \cite{Geymonat62, LionsMagenes61II, LionsMagenes62V, LionsMagenes63VI, Magenes65, MikhailetsMurach14, Murach09MFAT2}. If the functions $\alpha$ and $\eta$ are regularly varying at infinity, the space $H^{\alpha}_{A,\eta}(\Omega)$ is also applied to these problems (see \cite[Section~4.5.2]{MikhailetsMurach14} and \cite{AnopKasirenkoMurach18UMJ3}). Note that $H^{\alpha}_{A,\eta}(\Omega)$ may depend on each coefficient of the differential expression $A$, even when all these coefficients are constant. For instance, this is the case if $\alpha(t)\equiv\eta(t)\equiv1$ \cite[Theorem~3.1]{Hermander55}. Recall in this connection that the space $H^{0}_{A,0}(\Omega)$ is the domain of the maximal operator that corresponds to the unbounded operator $C^{\infty}(\overline{\Omega})\ni u\mapsto Au$ in $L_{2}(\Omega)$ (see, e.g., \cite{Hermander55}).

In the sequel, we suppose that $\varphi\in\mathrm{OR}$ and consider the case where $\sigma_0(\varphi)\leq-1/2$. Let us formulate our key result, which is a version of Proposition~\ref{prop1} in this case. We choose real numbers $s_0$, $s_1$, and $\lambda$ such that
\begin{equation}\label{f4.4}
s_0<\sigma_0(\varphi),\quad s_1>\sigma_1(\varphi),\quad \lambda>-1/2
\end{equation}
and that
\begin{equation}\label{f4.5}
\left\{
  \begin{array}{ll}
    \lambda\leq s_{1}&\hbox{if}\;\;\;\sigma_1(\varphi)\geq-1/2;\\
    s_{1}<-1/2&\hbox{if}\;\;\;\sigma_1(\varphi)<-1/2.
  \end{array}
\right.
\end{equation}

If $\sigma_1(\varphi)\geq-1/2$, we introduce the function
\begin{equation}\label{f13}
\eta(t):=t^{(1-\theta)s_{1}}\varphi(t^\theta)
\quad\mbox{of}\quad t\geq1,\quad\mbox{with}\quad \theta:=\frac{s_1-\lambda}{s_1-s_0}.
\end{equation}
Then $0\leq\theta<1$, $\eta\in\mathrm{OR}$, and $\sigma_{j}(\eta)=(1-\theta)s_{1}+\theta\sigma_{j}(\varphi)$ for every $j\in\{0,1\}$, which implies in view of \eqref{f4.4} that $\sigma_{0}(\eta)>\lambda>-1/2$ and, hence, $H^{\eta}(\Omega)\hookrightarrow H^{\lambda}(\Omega)$. Besides, since $\varphi(t)/\varphi(t^{\theta})\leq c_{1}t^{(1-\theta)s_{1}}$ whenever $t\geq1$ due to \eqref{f3.2}, we conclude that $\varphi(t)/\eta(t)\leq c_{1}$ whenever $t\geq1$, which implies the continuous embedding $H^{\eta}(\Omega)\hookrightarrow H^{\varphi}(\Omega)$.

If $\sigma_1(\varphi)<-1/2$, we put $\eta(t):=t^{\lambda}$ for every $t\geq1$. Then $H^{\eta}(\Omega)=H^{\lambda}(\Omega)\hookrightarrow H^{\varphi}(\Omega)$ because $\lambda>-1/2>\sigma_{1}(\varphi)$.

Thus,
\begin{equation}\label{f13bis}
\mbox{the continuous embedding}\quad H^{\eta}(\Omega)\hookrightarrow H^{\lambda}(\Omega)\cap H^{\varphi}(\Omega)
\end{equation}
holds true whatever $\sigma_1(\varphi)$ is.

The following theorem is a key result of this paper.

\begin{theorem}\label{th1}
The set $C^{\infty}(\overline{\Omega})$ is dense in the space  $H^{\varphi\rho^{2q}}_{A,\eta}(\Omega)$, and the mapping \eqref{mapping} extends uniquely (by continuity) to a bounded linear operator \begin{equation}\label{f15}
(A,B):H^{\varphi\rho^{2q}}_{A,\eta}(\Omega)\to H^{\eta}(\Omega)\oplus
\bigoplus_{j=1}^{q}H^{\varphi\rho^{2q-m_j-1/2}}(\Gamma)=:
\mathcal{H}^{\eta,\varphi}(\Omega,\Gamma).
\end{equation}
This operator is Fredholm. Its kernel coincides with $N$, and its range  consists of all vectors $(f,g)\in\mathcal{H}^{\eta,\varphi}(\Omega,\Gamma)$ that satisfy \eqref{f10}. The index of the operator \eqref{f15} equals $\dim N-\dim N^{+}$ and does not depend on $\varphi$ and~$\eta$.
\end{theorem}

This theorem and other results of this section will be proved in Section~\ref{sec6}.

\begin{remark}\label{bounded-operator}
If the system $B$ were not normal or if it did not satisfy the Lopatinskii condition, the operator \eqref{f15} would remain to be well defined and bounded. This follows from the fact that the boundedness of the operator \eqref{f9} does not depend on the ellipticity of the problem \eqref{f1}, \eqref{f2} (see the proof of Theorem~\ref{th1} in Section~\ref{sec6}).
\end{remark}

Let us discuss Theorem~\ref{th1} in the Sobolev case where $\varphi(t)\equiv t^{s}$ for a certain real number $s\leq-1/2$. If $s<-1/2$, this theorem yields the Fredholm bounded operator
\begin{equation}\label{f4.7}
(A,B):H^{s+2q}_{A,\lambda}(\Omega)\to H^{\lambda}(\Omega)\oplus
\bigoplus_{j=1}^{q}H^{s+2q-m_j-1/2}(\Gamma)=:
\mathcal{H}^{\lambda,s}(\Omega,\Gamma)
\end{equation}
for every real number $\lambda>-1/2$. If $s=-1/2$, we obtain the same operator by choosing $s_{0}=-1$ and $s_{1}=\lambda>-1/2$ in \eqref{f13}. The boundedness and Fredholm property of the operator \eqref{f4.7} were proved by Lions and Magenes \cite{LionsMagenes62V, LionsMagenes63VI} provided that $\lambda=0$, $s<-1/2$, and
\begin{equation}\label{f4.8}
s+2q\neq-k+1/2\quad\mbox{whenever}\quad 1\leq k\in\mathbb{Z}.
\end{equation}
Their result was extended to every $\lambda>-1/2$ in \cite[Corollary~1]{Murach09MFAT2} (see also \cite[Theorem~4.27]{MikhailetsMurach14}). If $\varphi$ is a regular varying function at infinity of order $s+2q<-1/2$ subject to \eqref{f4.8}, Theorem~\ref{th1} is established in \cite[Theorem~4.32]{MikhailetsMurach14}. If $\varphi\in\mathrm{OR}$, this theorem is proved in our paper \cite[Section~4]{AnopMurach15Coll2} in the case where $s_{1}\geq0$ and $\lambda=0$, the paper being published in Ukrainian. In this case, $H^{\eta}(\Omega)\subseteq L_{2}(\Omega)$.

Remark that Theorem~\ref{th1} remains true if we change $\eta$ for every $\omega\in\mathrm{OR}$ such that the function $\eta/\omega$ is bounded in a neighbourhood of infinity. This follows plainly from the continuous embedding $H^{\omega}(\Omega)\hookrightarrow H^{\eta}(\Omega)$.

If $N=\{0\}$ and $N^+=\{0\} $, the operator \eqref{f15} is an isomorphism between the spaces $H^{\varphi\rho^{2q}}_{A,\eta}(\Omega)$ and $\mathcal{H}^{\eta,\varphi}(\Omega,\Gamma)$. Generally, this operator induces an isomorphism between some of their subspaces of finite codimension. In this connection, the next result will be useful.

\begin{lemma}\label{lema1}
Let $\alpha,\omega\in\mathrm{OR}$ and $r\in\mathbb{R}$ satisfy $r<\sigma_{0}(\alpha)\leq0$ and $\sigma_{0}(\omega)>-1/2$. Then there exists a number $c>0$ such that
\begin{equation}\label{lema1-bound}
|(u,w)_{\Omega}|\leq c\,
\|u\|_{\alpha,A,\omega}\cdot\|w\|_{-r,\Omega}
\end{equation}
for arbitrary functions $u,w\in C^{\infty}(\overline{\Omega})$. Thus,
the sesquilinear form $(u,w)_\Omega$ of $u,w\in C^{\infty}(\overline{\Omega})$ extends uniquely (by continuity) over all $u\in H^{\alpha}_{A,\omega}(\Omega)$ and $w\in H^{-r}(\Omega)$.
\end{lemma}

\begin{remark}
If $\sigma_0(\alpha)<-1/2$, then we may not replace $H^{\alpha}_{A,\eta}(\Omega)$ with the broader space $H^{\alpha}(\Omega)$ in the last sentence of this lemma. If $H^{\alpha}(\Omega)$ is a Sobolev space, this follows from  \cite[Theorems 4.8.2(c) and 4.3.2/1(c)]{Triebel95}.
\end{remark}

Using Lemma \ref{lema1} in the $\sigma_{0}(\varphi)\leq-2q$ case and the continuous embedding $H^{\varphi\rho^{2q}}_{A,\eta}(\Omega)\hookrightarrow L_{2}(\Omega)$ otherwise, we may split the source space of the operator \eqref{f15} into the direct sum of subspaces
\begin{equation}\label{sum1}
H^{\varphi\rho^{2q}}_{A,\eta}(\Omega)=N\dotplus\bigl\{u\in
H^{\varphi\rho^{2q}}_{A,\eta}(\Omega):\,(u,w)_\Omega=0\;\;\mbox{for every}\;\;w\in N\bigr\}.
\end{equation}
Indeed, if $\sigma_{0}(\varphi)\leq-2q$, then Lemma~\ref{lema1} implies the continuous embedding of $H^{\varphi\rho^{2q}}_{A,\eta}(\Omega)$ in the dual of $H^{-r}(\Omega)$ provided that $r<2q+\sigma_{0}(\varphi)$. Since $N\subset H^{\varphi\rho^{2q}}_{A,\eta}(\Omega)\cap H^{-r}(\Omega)$, we decompose this dual into the direct sum
\begin{equation*}
(H^{-r}(\Omega))'=N\dotplus\bigl\{u\in(H^{-r}(\Omega))': \,u=0\;\,\mbox{on}\;\,N\bigr\}.
\end{equation*}
Note that the codimension of the second summand equals $\dim N'=\dim N<\infty$. The restriction of this decomposition to $H^{\varphi\rho^{2q}}_{A,\eta}(\Omega)$ gives \eqref{sum1}. If $\sigma_{0}(\varphi)>-2q$, then \eqref{sum1} is a restriction of the orthogonal sum
\begin{equation*}
L_{2}(\Omega)=N\oplus\bigl\{u\in L_{2}(\Omega):\, (u,w)_\Omega=0\;\;\mbox{for every}\;\;w\in N \,\bigr\}.
\end{equation*}

Besides, we may split the target space of the operator \eqref{f15} as follows:
\begin{equation}\label{sum2b}
\mathcal{H}^{\eta,\varphi}(\Omega,\Gamma)=
\bigl\{(w,0,\ldots,0):w\in N^+\bigr\}\dotplus
\bigl\{(f,g)\in\mathcal{H}^{\eta,\varphi}(\Omega,\Gamma):
\mbox{\eqref{f10} is true}\bigr\}.
\end{equation}
Indeed, $\mathcal{H}^{\eta,\varphi}(\Omega,\Gamma)\hookrightarrow H^{\ell}(\Omega)\oplus(H^{-r}(\Gamma))^{q}=:\Xi$ for $\ell:=\min\{\lambda,0\}\in(-1/2,0]$ and $r\gg1$. The latter space admits the decomposition
\begin{equation*}
\Xi=\bigl\{(w,0,\ldots,0):w\in N^+\bigr\}\dotplus
\bigl\{(f,g)\in\Xi:\mbox{\eqref{f10} is true}\bigr\}
\end{equation*}
because the codimension of the second summand equals $\dim M'=\dim N^+<\infty$, where $M:=\{(v,C_{1}^{+}v,\ldots,C_{q}^{+}v):v\in N^+\}$. Here, we consider $M$ as a subspace of $\Psi:=H^{-\ell}(\Omega)\oplus(H^{r}(\Gamma))^{q}$ and note that $\Xi$ is the dual of $\Psi$ with respect to the form $(\cdot,\cdot)_{\Omega}+(\cdot,\cdot)_{\Gamma}+\cdots+
(\cdot,\cdot)_{\Gamma}$. Now \eqref{sum2b} is a restriction of the above decomposition of $\Xi$.

Let $P$ and $\mathcal{P}^+$ respectively denote the projectors of the spaces $H^{\varphi\rho^{2q}}_{A,\eta}(\Omega)$ and $\mathcal{H}^{\eta,\varphi}(\Omega,\Gamma)$ onto the second summand in \eqref{sum1} and \eqref{sum2b} parallel to the first. The mappings  defining these projectors do not depend on $\varphi$ and $\eta$.

\begin{theorem}\label{th2}
The restriction of the operator \eqref{f15} to the second summand in \eqref{sum1} is an isomorphism
\begin{equation}\label{isom}
(A,B):P\bigl(H^{\varphi\rho^{2q}}_{A,\eta}(\Omega)\bigr)\leftrightarrow
\mathcal{P}^+\bigl(\mathcal{H}^{\eta,\varphi}(\Omega,\Gamma)\bigr).
\end{equation}
\end{theorem}

Let us now focus on properties of generalized solutions to the elliptic problem \eqref{f1}, \eqref{f2} in the extended Sobolev scale. Our definition of such solutions is suggested by Theorem~\ref{th1}.
We put
\begin{equation*}
\mathcal{S}'(\Omega):=\bigl\{w\!\upharpoonright\Omega\!:
w\in\mathcal{S}'(\mathbb{R}^{n})\bigr\}=
\bigcup_{r\in\mathbb{R}}H^{r}(\Omega)
\end{equation*}
and note that $\mathcal{D}'(\Gamma)$ coincides with the union of all spaces $H^{r}(\Gamma)$ where $r\in\mathbb{R}$. If $u\in\mathcal{S}'(\Omega)$ satisfies \eqref{f1} for certain $f\in H^{-1/2+}(\Omega)$, then $u\in H^{s+2q}_{A,\lambda}$ for some $s<-1/2$ and $\lambda>\nobreak -1/2$. Hence, the vector $g:=Bu\in(\mathcal{D}'(\Gamma))^{q}$ is well defined by closure due to Theorem~\ref{th1} considered in the Sobolev case. Therefore, conditions
\eqref{f1} and \eqref{f2} make sense provided that $u\in\mathcal{S}'(\Omega)$, $f\in H^{-1/2+}(\Omega)$, and $g\in(\mathcal{D}'(\Gamma))^{q}$. If $u\in\mathcal{S}'(\Omega)$ satisfies these conditions, we will call $u$ a generalized solution to the problem \eqref{f1}, \eqref{f2}.

\begin{theorem}\label{th4.6}
Assume that a distribution $u\in\mathcal{S}'(\Omega)$ is a generalized solution to the elliptic problem \eqref{f1}, \eqref{f2} whose right-hand sides satisfy the conditions $f\in H^{\eta}(\Omega)$ and $g_j\in H^{\varphi\rho^{2q-m_j-1/2}}(\Gamma)$ for each $j\in\{1,\ldots,q\}$. Then $u\in H^{\varphi\rho^{2q}}(\Omega)$.
\end{theorem}

Let us formulate a local version of this theorem. Let $U$ be an arbitrary open subset of $\mathbb{R}^{n}$ such that
$\Omega_0:=\Omega\cap U\neq\emptyset$ and $\Gamma_{0}:=\Gamma\cap U\neq\emptyset$. Given $\alpha\in\mathrm{OR}$, we let $H^{\alpha}_{\mathrm{loc}}(\Omega_{0},\Gamma_{0})$ denote the linear space of all distributions $u\in\mathcal{S}'(\Omega)$ such that $\chi u\in H^{\alpha}(\Omega)$ for every function $\chi\in C^{\infty}(\overline{\Omega})$ satisfying $\mathrm{supp}\,\chi\subset\Omega_0\cup\Gamma_{0}$. Analogously,
$H^{\alpha}_{\mathrm{loc}}(\Gamma_{0})$ denotes the linear space of all distributions $h\in\mathcal{D}'(\Gamma)$ such that $\chi h\in H^{\alpha}(\Gamma)$ for every function $\chi\in C^{\infty}(\Gamma)$ satisfying $\mathrm{supp}\,\chi\subset\Gamma_{0}$.

\begin{theorem}\label{th4.7}
Assume that a distribution $u\in\mathcal{S}'(\Omega)$ is a generalized solution to the elliptic problem \eqref{f1}, \eqref{f2} whose right-hand sides satisfy the conditions
\begin{equation}\label{f4.14}
f\in H^\eta_{\mathrm{loc}}(\Omega_{0},\Gamma_{0})\cap H^{-1/2+}(\Omega)
\end{equation}
and
\begin{equation}\label{f4.15}
g_j\in H^{\varphi\rho^{2q-m_j-1/2}}_{\mathrm{loc}}(\Gamma_0)
\quad\mbox{for each}\quad j\in\{1,\ldots,q\}.
\end{equation}
Then $u\in H^{\varphi\rho^{2q}}_{\mathrm{loc}}(\Omega_{0},\Gamma_{0})$.
\end{theorem}

\begin{remark}\label{rem4.8}
The given definition of $H^{\alpha}_{\mathrm{loc}}(\Omega_{0},\Gamma_{0})$ is also applicable to the $\Gamma_{0}=\emptyset$ case. As to condition \eqref{f4.14}, note that if $u\in\mathcal{D}'(\Omega)$ and $f:=Au\in H^\eta_{\mathrm{loc}}(\Omega_{0},\emptyset)$, then $u\in H^{\eta\varrho^{2q}}_{\mathrm{loc}}(\Omega_{0},\emptyset)$ according to
\cite[Theorem 7.4.1]{Hermander63}.
\end{remark}

As an application of Theorem~\ref{th4.7}, we give sufficient conditions for generalized derivatives (of a given order) of the solution $u$ to be continuous on $\Omega_{0}\cup\Gamma_{0}$. Assuming $0\leq p\in\mathbb{Z}$ and $u\in\mathcal{D}'(\Omega)$, we write $u\in C^{p}(\Omega_{0}\cup\Gamma_{0})$ if there exists a function $u_0\in C^{p}(\Omega_{0}\cup\Gamma_{0})$ such that
\begin{equation}\label{def-C^p}
\bigl(v\in C^{\infty}_{0}(\Omega),\;
\mathrm{supp}\,v\subset\Omega_{0}\bigr)\Longrightarrow
(u,v)_{\Omega}=\int\limits_{\Omega_{0}}
u_0(x)\overline{v(x)}dx;
\end{equation}
here, $(u,v)_{\Omega}$ is the value of the distribution $u$ at the test function $v$.

\begin{theorem}\label{th4.9}
Let $0\leq p\in\mathbb{Z}$,
\begin{equation}\label{int-cond}
\int\limits_1^{\infty}t^{2p-4q+n-1}\varphi^{-2}(t)dt<\infty,
\end{equation}
and assume that a distribution $u\in\mathcal{S}'(\Omega)$ satisfies the hypotheses of Theorem~$\ref{th4.7}$. Then $u\in C^{p}(\Omega_{0}\cup\Gamma_{0})$.
\end{theorem}

\begin{remark}\label{rem4.10}
Condition \eqref{int-cond} is sharp in Theorem~\ref{th4.9}. Namely, let $0\leq p\in\mathbb{Z}$; it follows then from the implication
\begin{equation}\label{implication}
\begin{aligned}
&\bigl(\mbox{a distribution $u\in\mathcal{S}'(\Omega)$ satisfies the hypotheses of Theorem~\ref{th4.7}}\bigr)\\
&\Longrightarrow\;
u\in C^{p}(\Omega_{0}\cup\Gamma_{0})
\end{aligned}
\end{equation}
that $\varphi$ satisfies \eqref{int-cond}. This will be shown in Section~\ref{sec6}
\end{remark}

\begin{remark}\label{rem4.11}
Theorems \ref{th2}--\ref{th4.7} and \ref{th4.9} remain valid for every function parameter $\varphi\in\mathrm{OR}$ subject to $\sigma_0(\varphi)>-1/2$ provided that we put $\eta:=\varphi$. In this case, they relates to Proposition~\ref{prop1} and are demonstrated in the same way as the corresponding proofs given in Section~\ref{sec6}, we taking into account that $H^{\varphi\varrho^{2q}}_{A,\varphi}(\Omega)=
H^{\varphi\varrho^{2q}}(\Omega)$ up to equivalence of norms and supposing that $s<-1/2$ in the proofs. These theorems are proved in \cite{AnopMurach14UMJ} in the case indicated, with the assumption $u\in H^{2q-1/2+}(\Omega)$ being made instead of $u\in\mathcal{S}'(\Omega)$.
\end{remark}

We supplement Theorems \ref{th4.6} and \ref{th4.7} with \textit{a priori} estimates of the solution $u$. Let $\|(f,g)\|_{\eta,\varphi,\Omega,\Gamma}$ denote the norm of a vector $(f,g)=(f,g_1,\ldots,g_q)$ in the Hilbert space
$\mathcal{H}^{\eta,\varphi}(\Omega,\Gamma)$ defined in \eqref{f15}.

\begin{theorem}\label{th4.12}
Assume that a distribution $u\in\mathcal{S}'(\Omega)$ satisfies the hypotheses of Theorem~$\ref{th4.6}$, and choose a number $\ell>0$ arbitrarily. Then
\begin{equation}\label{f4.20}
\|u\|_{\varphi\rho^{2q},\Omega}\leq c\,\bigl(\|(f,g)\|_{\eta,\varphi,\Omega,\Gamma}+
\|u\|_{\varphi\rho^{2q-\ell},\Omega}\bigr)
\end{equation}
for some number $c>0$ that does not depend on $u$ and $(f,g)$.
\end{theorem}

Note, if $N=\{0\}$, we may remove the last summand on the right of \eqref{f4.20} due to the Banach theorem on inverse operator.

A local version of this result is stated as follows:

\begin{theorem}\label{th4.13}
Assume that a distribution $u\in\mathcal{S}'(\Omega)$ satisfies the hypotheses of Theorem~$\ref{th4.7}$. We arbitrarily choose a number $\ell>0$ and functions $\chi,\zeta\in C^{\infty}(\overline{\Omega})$ such that $\mathrm{supp}\,\chi\subset\mathrm{supp}\,\zeta\subset
\Omega_{0}\cup\Gamma_{0}$ and that $\zeta=1$ in a neighbourhood of $\mathrm{supp}\,\chi$. Then
\begin{equation}\label{f4.21}
\|\chi u\|_{\varphi\rho^{2q},\Omega}\leq c\,\bigl(\|\zeta(f,g)\|_{\eta,\varphi,\Omega,\Gamma}+
\|\zeta u\|_{\varphi\rho^{2q-\ell},\Omega}\bigr)
\end{equation}
for some number $c>0$ that does not depend on $u$ and $(f,g)$.
\end{theorem}

\begin{remark}\label{rem4.14}
Theorems \ref{th4.12} and \ref{th4.13} remain valid for every  $\varphi\in\mathrm{OR}$ subject to $\sigma_0(\varphi)>-1/2$ provided that we put $\eta:=\varphi$. In this case, they relate to Proposition~\ref{prop1} and are proved in the same way as that given in Section~\ref{sec6}, the proof of the version of Theorem~\ref{th4.13} being simplified (see Remark~\ref{rem6.3} at the end of Section~\ref{sec6}). Theorem \ref{th4.13} is proved in \cite[Theorem~3]{AnopKasirenko16MFAT} in the case indicated and under the assumptions that $u\in H^{\varphi\rho^{2q}}(\Omega)$ and $l=1$.
\end{remark}

As far as we know, Theorems \ref{th4.7} and \ref{th4.13} are new even in the Sobolev case where $\varphi(t)\equiv t^{s}$ and $\eta(t)\equiv t^{\lambda}$ for some $s\leq-1/2$ and $\lambda>-1/2$.

\section{Interpolation properties of the extended Sobolev scale}\label{sec5}

The method of interpolation with a function parameter between Hilbert spaces play a crucial role in our proof of the key Theorem~\ref{th1}. Therefore, it is worthwhile to recall the definition of this method. We also discuss some properties of the extended Sobolev scale that relate to the method and will be used in our proofs. This method was appeared first in Foia\c{s} and Lions' article \cite[p.~278]{FoiasLions61}. We will mainly follow the monograph \cite[Section~1.1]{MikhailetsMurach14}.

Let $X:=[X_{0},X_{1}]$ be an ordered pair of separable complex Hilbert spaces such that $X_{1}$ is a linear manifold in $X_{0}$ and that $\|w\|_{X_{0}}\leq c\|w\|_{X_{1}}$ for a certain number $c>0$ and every vector $w\in X_{1}$. This pair is called regular. For this pair there exists a unique positive-definite self-adjoint operator $J$ acting in $X_{0}$, defined on $X_{1}$, and obeying $\|Jw\|_{X_{0}}=\|w\|_{X_{1}}$ whenever $w\in X_{1}$. This operator is called a generating operator for~$X$.

Consider a Borel measurable function $\psi:(0,\infty)\rightarrow(0,\infty)$ such that $\psi$ is bounded on each compact subset of $(0,\infty)$ and that $1/\psi$ is bounded on every set $[r,\infty)$, with $r>0$. The class of all such functions $\psi$ is denoted by $\mathcal{B}$. Using the spectral resolution of $J$, we get the positive-definite self-adjoint operator $\psi(J)$ in $X_{0}$. Let
$[X_{0},X_{1}]_{\psi}$ or, simply, $X_{\psi}$ denote the domain of
$\psi(J)$ endowed with the inner product
$(w_1, w_2)_{X_\psi}:=(\psi(J)w_1,\psi(J)w_2)_{X_0}$ and the
corresponding norm $\|w\|_{X_\psi}=(w,w)_{X_\psi}^{1/2}$. The space $X_{\psi}$ is Hilbert and separable.

A function $\psi\in\mathcal{B}$ is called an interpolation parameter if the following condition is fulfilled for all regular pairs $X=[X_{0},X_{1}]$ and $Y=[Y_{0},Y_{1}]$ of Hilbert spaces and for any linear mapping $T$ given on $X_{0}$: if the restriction of $T$ to $X_{j}$ is a bounded operator $T:X_{j}\rightarrow Y_{j}$ for each $j\in\{0,1\}$, then the restriction of $T$ to $X_{\psi}$ is also a bounded operator $T:X_{\psi}\rightarrow Y_{\psi}$. If $\psi$ is an interpolation parameter, we will say that the Hilbert space $X_{\psi}$ is obtained by the interpolation with the function parameter $\psi$ between $X_{0}$ and $X_{1}$ and that the bounded operator $T:X_{\psi}\rightarrow Y_{\psi}$ is obtained by the interpolation of the operators $T:X_{j}\rightarrow Y_{j}$ with $j\in\{0,1\}$. In this case, we have the dense continuous embeddings
$X_{1}\hookrightarrow X_{\psi}\hookrightarrow X_{0}$.

The function $\psi$ is an interpolation parameter if and only
if $\psi$ is pseudoconcave in a neighbourhood of $+\infty$. The latter condition means that there exists a concave function $\psi_{1}:(b,\infty)\rightarrow(0,\infty)$, with $b\gg1$, that the
functions $\psi/\psi_{1}$ and $\psi_{1}/\psi$ are bounded on $(b,\infty)$.  This fundamental result follows from Peetre's \cite{Peetre68} description  of all interpolation functions of positive order (see  \cite[Sect.~1.1.9]{MikhailetsMurach14}.

\begin{proposition}\label{prop5.1}
Let $\alpha\in\mathrm{OR}$, $r_{0},r_{1}\in\mathbb{R}$, $r_{0}<\sigma_{0}(\alpha)$, and $r_{1}>\sigma_{1}(\alpha)$. Put
\begin{equation}\label{f5.1}
\psi(t):=
\begin{cases}
\;t^{{-r_0}/{(r_1-r_0)}}\,
\alpha\bigl(t^{1/{(r_1-r_0)}}\bigr)&\text{if}\quad t\geq1; \\
\;\alpha(1)&\text{if}\quad0<t<1.
\end{cases}
\end{equation}
Then $\psi\in\mathcal{B}$ is an interpolation parameter, and
\begin{equation*}
[H^{r_0}(G),H^{r_1}(G)]_{\psi}=H^{\alpha}(G)
\end{equation*}
with equality of norms if $G=\mathbb{R}^{n}$, and with equivalence of norms if $G=\Omega$ or $G=\Gamma$.
\end{proposition}

This result is proved in \cite[Theorem~5.1]{MikhailetsMurach15ResMath1} for $G=\Omega$ and in \cite[Theorems 2.19 and 2.22]{MikhailetsMurach14} for $G\in\{\mathbb{R}^{n},\Gamma\}$.

The next result shows that the extended Sobolev scale is closed with respect to this interpolation.

\begin{proposition}\label{prop5.2}
Let $\alpha_0,\alpha_1\in\mathrm{OR}$ and $\psi\in\mathcal{B}$. Assume that the function $\alpha_0/\alpha_1$ is bounded in a neighbourhood of infinity and that $\psi$ is an interpolation parameter. Put $\alpha(t):=\alpha_0(t)\psi(\alpha_1(t)/\alpha_0(t))$ for every $t\geq1$. Then $\alpha\in\mathrm{OR}$, and
\begin{equation*}
[H^{\alpha_0}(G),H^{\alpha_1}(G)]_{\psi}=H^{\alpha}(G)
\end{equation*}
with equality of norms if $G=\mathbb{R}^{n}$, and with equivalence of norms if $G=\Omega$ or $G=\Gamma$.
\end{proposition}

This result is proved in \cite[Theorem~5.2]{MikhailetsMurach15ResMath1} for $G=\Omega$ and in \cite[Theorems 2.18 and 2.22]{MikhailetsMurach14} for $G\in\{\mathbb{R}^{n},\Gamma\}$. We will use it in the case where $\alpha_0$ and $\alpha_1$ are power functions.

It follows from Proposition~\ref{prop5.1} and Ovchinnikov's theorem \cite[Theorem 11.4.1]{Ovchinnikov84} that this scale coincides (up to equivalence of norms) with the class of all Hilbert
spaces that are interpolation ones between the Sobolev
spaces $H^{(r_0)}(G)$ and $H^{(r_1)}(G)$ where $r_0,r_1\in\mathbb{R}$
and $r_0<r_1$ (as above, $G\in\{\mathbb{R}^{n},\Omega,\Gamma\}$). Recall that a Hilbert space $H$ is called an interpolation space between $X_0$ and $X_1$ if the following two properties are satisfied: a)~the continuous embeddings $X_1\hookrightarrow H\hookrightarrow X_0$ hold; b)~every linear operator bounded on $X_0$ and $X_1$ should be also bounded on $H$.

Let us establish a version of Proposition~\ref{prop5.1} for the space $H^{\alpha}_{A,\eta}(\Omega)$.

\begin{theorem}\label{th5.3}
Let $\alpha\in\mathrm{OR}$ and suppose that real numbers $r_{0}$, $r_{1}$, $\lambda_{0}$, and $\lambda_{1}$ satisfy the conditions $r_{0}<\sigma_{0}(\alpha)$, $r_{1}>\sigma_{1}(\alpha)$, $\lambda_{0}\leq\lambda_{1}$, $\lambda_{0}\geq r_{0}-2q$, and $\lambda_{1}\geq r_{1}-2q$. Besides, let $\psi$ be the interpolation parameter from Proposition~$\ref{prop5.1}$. Then the pair of separable Hilbert spaces $H^{r_0}_{A,\lambda_{0}}(\Omega)$ and $H^{r_1}_{A,\lambda_{1}}(\Omega)$ is regular, and
\begin{equation}\label{f5.2}
\bigl[H^{r_0}_{A,\lambda_{0}}(\Omega),
H^{r_1}_{A,\lambda_{1}}(\Omega)\bigr]_{\psi}=H^{\alpha}_{A,\chi}(\Omega)
\end{equation}
up to equivalence of norms. Here, the function $\chi\in\mathrm{OR}$ is defined as follows:
\begin{equation}\label{f5.3}
\chi(t):=t^{\lambda_{0}}\psi\bigl(t^{\lambda_{1}-\lambda_{0}}\bigr)=
t^{(r_1\lambda_0-r_0\lambda_1)/(r_1-r_0)}
\alpha\bigl(t^{(\lambda_1-\lambda_0)/(r_1-r_0)}\bigr)
\end{equation}
for every $t\geq1$.
\end{theorem}

\begin{remark}\label{rem5.4}
The case $\lambda_{0}=\lambda_{1}=:\lambda$ gives $\chi(t)\equiv t^{\lambda}\alpha(1)$ by \eqref{f5.3}. Hence, formula \eqref{f5.2} remains true in this case if we put $\chi(t):=t^{\lambda}$ whenever $t\geq1$.
\end{remark}

\begin{remark}\label{rem5.5}
If $\lambda_{j}=r_{j}-2q$ for certain $j\in\{0,1\}$, then $H^{r_j}_{A,\lambda_{j}}(\Omega)=H^{r_j}(\Omega)$ up to equivalence of norms. This follows from the boundedness of the operator $A:H^{r_j}(\Omega)\to H^{r_j-2q}(\Omega)$. Hence, the interpolation formula \eqref{f5.2} is applicable to the case where the Sobolev space $H^{r_j}(\Omega)$ is taken instead of $H^{r_j}_{A,\lambda_{j}}(\Omega)$.
\end{remark}

The proof of this theorem is based on Proposition~\ref{prop5.1} and
a result \cite[Theorem~3.12]{MikhailetsMurach14} on interpolation with a function parameter between certain Hilbert spaces induced by a bounded linear operator. Before we formulate this result, let us admit the following: if $H$, $\Phi$ and $\Psi$ are Hilbert spaces satisfying the continuous embedding $\Phi\hookrightarrow\Psi$ and if $T:H\rightarrow\Psi$ is a continuous linear operator, we put
$$
(H)_{T,\Phi}:=\{u\in H:\,Tu\in\Phi\}
$$
and endow the linear space $(H)_{T,\Phi}$ with the inner product
$$
(u_1,u_2)_{(H)_{T,\Phi}}:=(u_1,u_2)_{H}+(Tu_1,Tu_2)_{\Phi}
$$
and the corresponding norm $\|u\|_{(H)_{T,\Phi}}:=(u,u)_{(H)_{T,\Phi}}^{1/2}$. The inner product does not depend on $\Psi$, and the space $(H)_{T,\Phi}$ is Hilbert. The latter is proved in a quite similar way as the proof of the completeness of $H^{\alpha}_{A,\eta}(\Omega)$ given in Section~\ref{sec4} just after~\eqref{f11}.

\begin{proposition}\label{prop5.6}
Assume that six separable Hilbert spaces $X_{0}$, $Y_{0}$, $Z_{0}$, $X_{1}$, $Y_{1}$, and $Z_{1}$ and three linear mappings $T$, $R$, and $S$ are given and satisfy the following conditions:
\begin{itemize}
\item[(i)] The pairs $[X_{0},X_{1}]$ and  $[Y_{0},Y_{1}]$ are regular.
\item[(ii)] The spaces $Z_{0}$ and $Z_{1}$ are subspaces of a certain linear space $E$.
\item[(iii)] The continuous embeddings $Y_{0}\hookrightarrow Z_{0}$ and $Y_{1}\hookrightarrow Z_{1}$ hold.
\item[(iv)] The mapping $T$ is given on $X_{0}$ and defines the bounded operators $T:\nobreak X_{0}\rightarrow Z_{0}$ and $T:X_{1}\rightarrow Z_{1}$.
\item[(v)] The mapping $R$ is given on $E$ and defines the bounded operators $R:Z_{0}\rightarrow X_{0}$ and $R:Z_{1}\rightarrow X_{1}$.
\item[(vi)] The mapping $S$ is given on $E$ and defines the bounded operators $S:Z_{0}\rightarrow Y_{0}$ and
    $S:Z_{1}\rightarrow Y_{1}$.
\item[(vii)] The equality $TRu=u+Su$ holds for every $u\in E$.
\end{itemize}
Then the pair of the separable Hilbert spaces $(X_{0})_{T,Y_{0}}$ and $(X_{1})_{T,Y_{1}}$ is regular, and
\begin{equation}
\bigl[(X_{0})_{T,Y_{0}},(X_{1})_{T,Y_{1}}\bigr]_{\psi}=
(X_{\psi})_{T,Y_{\psi}}
\end{equation}
up to equivalence of norms for every interpolation parameter $\psi\in\mathcal{B}$.
\end{proposition}

Note that conditions (i)--(vii) were found by Lions and Magenes \cite[Chapter~1, Theorem~14.3]{LionsMagenes72}, who proved a version of Proposition~\ref{prop5.6} for the holomorphic interpolation (with a number parameter).

\begin{proof}[Proof of Theorem~$\ref{th5.3}$.]
Choosing an integer $p\geq1$ arbitrarily, we consider the
linear PDO $A^{p}A^{p+}+I$ of order $4qp$. Here, as usual, $A^{p+}$ denotes the PDO which is formally adjoint to the $p$-th iteration $A^p$ of $A$, and $I$ is the identity operator. Let $H_D^\sigma(\Omega)$, where $\sigma\geq2qp$, denote the set of all distributions $u\in H^\sigma(\Omega)$  such that $\partial_\nu^ju=0$ on $\Gamma$ for each $j\in\{0,\ldots,2qp-1\}$, with $\partial_\nu$ being the operator of the differentiation with respect to the inward normal to the boundary $\Gamma$ of $\Omega$. The linear manifold $H_D^\sigma(\Omega)$  is well defined and closed in $H^\sigma(\Omega)$ due to the theorem on traces for Sobolev spaces (see, e.g., \cite[Section 4.7.1]{Triebel95}). We hence consider $H_D^\sigma(\Omega)$ as a  subspace of $H^\sigma(\Omega)$. The differential operator $A^{p}A^{p+}+I$ sets an isomorphism
$$
A^pA^{p+}+I:H_D^\sigma(\Omega)\leftrightarrow H^{\sigma-4qp}(\Omega)
$$
for each integer $\sigma\geq2qp$ (see, e.g., \cite[Lemma~3.1]{MikhailetsMurach14}). The inverse of this isomorphism sets a bounded linear operator
\begin{equation}\label{f524}
(A^pA^{p+}+I)^{-1}:H^l(\Omega)\rightarrow H^{l+4qp}(\Omega)
\end{equation}
for each integer $l\geq-2qp$. It follows from Proposition~\ref{prop5.1} that this operator is well defined and continuous for every real $l\geq -2qp$.

Turning to Proposition \ref{prop5.6}, we put $X_j:= H^{r_j}(\Omega)$, $Y_j:= H^{\lambda_j}(\Omega)$, and $Z_j :=H^{r_j-2q}(\Omega)$ for each $j\in\{0,1\}$, and $E:=H^{r_0-2q}(\Omega)$ and $T:=A$. Conditions (i)–(iv) of this proposition are evidently satisfied. We subject $p$ to the restrictions $r_j-2q\geq-2qp$ and $r_j-2q-\lambda_j\geq-4qp$ for each $j\in\{0,1\}$ and put
$$
R:= A^{p-1} A^{p+}(A^p A^{p+}+I)^{-1}\quad\text{and}\quad
S:=-(A^p A^{p+}+I)^{-1}.
$$
According to \eqref{f524}, we have the continuous linear operators
$$
R: Z_j = H^{r_j-2q}(\Omega)\rightarrow H^{r_j}(\Omega)= X_j
$$
and
$$
S: Z_j= H^{r_j-2q}(\Omega)\rightarrow H^{r_j-2q+4qp}(\Omega) \hookrightarrow H^{\lambda_j}(\Omega)=Y_j
$$
for each $j\in\{0,1\}$. Thus, conditions (v) and (vi) are also satisfied. The last condition (vii) is satisfied because
$$
T Ru = (A^p A^{p+}+I-I)(A^p A^{p+}+I)^{-1}u=u+Su
$$
for every $u\in E$.

Using Propositions \ref{prop5.6}, \ref{prop5.1}, and \ref{prop5.2} successively, we conclude that
\begin{align*}
\bigl[H^{r_0}_{A,\lambda_{0}}(\Omega),
H^{r_1}_{A,\lambda_{1}}(\Omega)\bigr]_{\psi}&=
\bigl[(X_{0})_{T,Y_{0}},(X_{1})_{T,Y_{1}}\bigr]_{\psi}=
(X_{\psi})_{T,Y_{\psi}}\\
&= \bigl([H^{r_0}(\Omega),H^{r_1}(\Omega)]_\psi\bigr)_
{A,[H^{\lambda_0}(\Omega),H^{\lambda_1}(\Omega)]_{\psi}}\\
&=\bigl(H^{\alpha}(\Omega)\bigr)_{A,H^{\chi}(\Omega)}
=H^{\alpha}_{A,\chi}(\Omega)
\end{align*}
up to equivalence of norms, which proves Theorem~$\ref{th5.3}$.
\end{proof}

Note that a version of Theorem $\ref{th5.3}$ for the interpolation with a number parameter is proved in \cite[Theorem~2]{KasirenkoMikhailetsMurach19}.

\section{Proofs of the main results}\label{sec6}

We will prove Theorem~$\ref{th1}$ with the help of its version for Sobolev spaces.

\begin{proposition}\label{prop6.1}
Let $s<-1/2$ and $\lambda>-1/2$. Then the set $C^{\infty}(\overline{\Omega})$ is dense in $H^{s+2q}_{A,\lambda}(\Omega)$, and the mapping \eqref{mapping} extends uniquely (by continuity) to a bounded linear operator
\begin{equation}\label{f6.1}
(A,B):H^{s+2q}_{A,\lambda}(\Omega)\to
H^{\lambda}(\Omega)\oplus\bigoplus_{j=1}^{q}H^{s+2q-m_j-1/2}(\Gamma)=
\mathcal{H}^{\lambda,s}(\Omega,\Gamma).
\end{equation}
This operator is Fredholm. Its kernel coincides with $N$, the range consists of all vectors $(f,g_{1},\ldots,g_{q})\in\mathcal{H}^{\lambda,s}(\Omega,\Gamma)$ that satisfy \eqref{f10}, and the index equals $\dim N-\dim N^{+}$.
\end{proposition}

\begin{proof}
This proposition is proved in \cite[Section~4.4.3]{MikhailetsMurach14} in the case \eqref{f4.8}. Let us examine the opposite case; i.e., we assume that $s+2q=-k+1/2$ for a certain integer $k\geq1$ and deduce  Proposition~\ref{prop6.1} from the first case with the help of the interpolation.

Choose numbers $s_0$ and $s_1$ so that $s_0<s<s_1<-1/2$ and that $s_j+2q-1/2\notin\mathbb{Z}$ whenever $j\in\{0,1\}$. We have the Fredholm bounded operators
\begin{equation}\label{f6.2}
(A,B):H^{s_j+2q}_{A,\lambda}(\Omega)\to
\mathcal{H}^{\lambda,s_j}(\Omega,\Gamma)
\quad\mbox{for each}\quad j\in\{0,1\}.
\end{equation}
Put $\alpha(t):=t^{s+2q}$ whenever $t\geq1$ and $r_{j}:=s_{j}+2q$ whenever $j\in\{0,1\}$, and define an interpolation parameter $\psi$ by formula \eqref{f5.1}. Thus, $\psi(t):=t^{(s-s_0)/(s_1-s_0)}$ for every $t\geq1$. Then a restriction of the operator \eqref{f6.2} for $j=0$ is a bounded operator between the spaces
\begin{equation}\label{f6.3}
(A,B):\bigl[H^{s_0+2q}_{A,\lambda}(\Omega),
H^{s_1+2q}_{A,\lambda}(\Omega)\bigr]_{\psi}\to
\bigl[\mathcal{H}^{\lambda,s_0}(\Omega,\Gamma),
\mathcal{H}^{\lambda,s_1}(\Omega,\Gamma)\bigr]_{\psi}.
\end{equation}
This operator is Fredholm according to the theorem on interpolation of Fredholm operators (see, e.g., \cite[Theorem~1.7]{MikhailetsMurach14}).

Owing to Theorem~\ref{th5.3} and in view of Remark~\ref{rem5.4}, we get
\begin{equation}\label{f6.4}
\bigl[H^{s_0+2q}_{A,\lambda}(\Omega),
H^{s_1+2q}_{A,\lambda}(\Omega)\bigr]_{\psi}=H^{s+2q}_{A,\lambda}(\Omega).
\end{equation}
Besides,
\begin{align*}
&\bigl[\mathcal{H}^{\lambda,s_0}(\Omega,\Gamma),
\mathcal{H}^{\lambda,s_1}(\Omega,\Gamma)\bigr]_{\psi}\\
&=\bigl[H^{\lambda}(\Omega),H^{\lambda}(\Omega)\bigr]_\psi
\oplus\bigoplus_{j=1}^{q}
\bigl[H^{s_0+2q-m_j-1/2}(\Gamma),H^{s_1+2q-m_j-1/2}(\Gamma)\bigr]_\psi\\
&=\mathcal{H}^{\lambda,s}(\Omega,\Gamma)
\end{align*}
due to Proposition~\ref{prop5.2} and the theorem on interpolation of orthogonal sums of Hilbert spaces (see, e.g., \cite[Theorem~1.5]{MikhailetsMurach14}). Hence, the Fredholm bounded operator \eqref{f6.3} acts between the spaces \eqref{f6.1}. Since the operators \eqref{f6.2} have the common kernel $N$ and index $\dim N-\dim N^{+}$, so does the operator \eqref{f6.1} according to \cite[Theorem~1.7]{MikhailetsMurach14}. Moreover,
\begin{align*}
(A,B)\bigl(H^{s+2q}_{A,\lambda}(\Omega)\bigr)&=
\mathcal{H}^{\lambda,s}(\Omega,\Gamma)\cap
(A,B)\bigl(H^{s_0+2q}_{A,\lambda}(\Omega)\bigr)\\
&=\bigl\{(f,g)\in\mathcal{H}^{\lambda,s}(\Omega,\Gamma):
\mbox{\eqref{f10} is satisfied}\bigr\}
\end{align*}
due to the same theorem.

Ending this proof, note that \eqref{f6.4} implies the dense continuous embedding of $H^{s_1+2q}_{A,\lambda}(\Omega)$ in $H^{s+2q}_{A,\lambda}(\Omega)$. Hence, the set $C^{\infty}(\overline{\Omega})$ being dense in the first space is also dense in the second.
\end{proof}

\begin{proof}[Proof of Theorem~$\ref{th1}$]
It is convenient to consider the cases where $\sigma_{1}(\varphi)\geq-1/2$ and were $\sigma_{1}(\varphi)<-1/2$ separately.

We assume first that $\sigma_{1}(\varphi)\geq-1/2$. Since $s_{1}>\sigma_{1}(\varphi)$, the mapping \eqref{mapping} extends uniquely (by continuity) to a bounded operator
\begin{equation}\label{f18}
(A,B):H^{s_{1}+2q}(\Omega)\rightarrow
H^{s_{1}}(\Omega)\oplus\bigoplus_{j=1}^{q}
H^{s_{1}+2q-m_j-1/2}(\Gamma)=:\mathcal{H}^{s_{1}}(\Omega,\Gamma).
\end{equation}
This operator is Fredholm with kernel $N$ and index $\dim N-\dim N^{+}$. This fact is a specific case of Proposition~\ref{prop1} and is well known in the $s_{1}\geq0$ case (see., e.g., \cite[Chapter~2, Section~5.4]{LionsMagenes72}). Besides, since $s_{0}<-1/2$ and $\lambda>-1/2$, we have the Fredholm bounded operator \eqref{f6.1} for $s:=s_0$, due to Proposition~\ref{prop6.1}. The kernel and index of \eqref{f6.1} are the same as those of \eqref{f18}. We put $\alpha(t):=\varphi(t)t^{2q}$ whenever $t\geq1$, set $r_{j}:=s_{j}+2q$ whenever $j\in\{0,1\}$, and define an interpolation parameter $\psi$ by formula \eqref{f5.1}. A restriction of \eqref{f6.1} is a bounded operator between the spaces
\begin{equation}\label{f6.6}
(A,B):\bigl[H^{s_0+2q}_{A,\lambda}(\Omega),
H^{s_1+2q}(\Omega)\bigr]_{\psi}\to \bigl[\mathcal{H}^{\lambda,s_0}(\Omega,\Gamma),
\mathcal{H}^{s_1}(\Omega,\Gamma)\bigr]_{\psi}.
\end{equation}
This operator is Fredholm with the same kernel and index according to \cite[Theorem~1.7]{MikhailetsMurach14}.

Owing to Theorem~\ref{th5.3} and in view of Remark~\ref{rem5.5}, we get
\begin{equation}\label{f6.7}
\bigl[H^{s_0+2q}_{A,\lambda}(\Omega),H^{s_1+2q}(\Omega)\bigr]_{\psi}=
\bigl[H^{s_0+2q}_{A,\lambda}(\Omega),
H^{s_1+2q}_{A,s_1}(\Omega)\bigr]_{\psi}=
H^{\varphi\varrho^{2q}}_{A,\eta}(\Omega).
\end{equation}
Note that $\eta(t)$ defined by \eqref{f13} is equal to $\chi(t)$ defined by \eqref{f5.3} if $\lambda_{0}=\lambda$ and $\lambda_{1}=s_1$. Indeed,
\begin{equation}\label{f6.7b}
\begin{aligned}
\chi(t)&=t^{\lambda}\psi(t^{s_1-\lambda})=
t^{\lambda}t^{-(s_0+2q)(s_1-\lambda)/(s_1-s_0)}
\alpha(t^{(s_1-\lambda)/(s_1-s_0)})\\
&=t^{\lambda-(s_0+2q)\theta}\alpha(t^{\theta})=
t^{\lambda-(s_0+2q)\theta}\varphi(t^{\theta})t^{2q\theta}=
t^{\lambda-s_0\theta}\varphi(t^{\theta})\\
&=t^{(1-\theta)s_1}\varphi(t^{\theta})=\eta(t)
\end{aligned}
\end{equation}
whenever $t\geq1$. Besides,
\begin{align*}
&\bigl[\mathcal{H}^{\lambda,s_0}(\Omega,\Gamma),
\mathcal{H}^{s_1}(\Omega,\Gamma)\bigr]_{\psi}\\
&=\bigl[H^{\lambda}(\Omega),H^{s_1}(\Omega)\bigr]_\psi
\oplus\bigoplus_{j=1}^{q}
\bigl[H^{s_0+2q-m_j-1/2}(\Gamma),H^{s_1+2q-m_j-1/2}(\Gamma)\bigr]_\psi\\
&=\mathcal{H}^{\eta,\varphi}(\Omega,\Gamma)
\end{align*}
due to Proposition~\ref{prop5.2}. Note here that $\eta(t)\equiv t^{\lambda}\psi(t^{s_1-\lambda})$ as was just shown and that
\begin{equation*}
t^{s_0+2q-m_j-1/2}\,\psi(t^{s_1-s_0})=
t^{s_0+2q-m_j-1/2}\,t^{-s_0-2q}\alpha(t)=\varphi(t)t^{2q-m_j-1/2}
\end{equation*}
whenever $t\geq1$.

Thus, the Fredholm bounded operator \eqref{f6.6} acts between the spaces \eqref{f15}. According to \cite[Theorem~1.7]{MikhailetsMurach14} and Proposition~\ref{prop6.1}, we conclude that
\begin{equation}\label{f6.8}
\begin{aligned}
(A,B)\bigl(H^{\varphi\varrho^{2q}}_{A,\eta}(\Omega)\bigr)&=
\mathcal{H}^{\eta,\varphi}(\Omega,\Gamma)\cap
(A,B)\bigl(H^{s_0+2q}_{A,\lambda}(\Omega)\bigr)\\
&=\bigl\{(f,g)\in
\mathcal{H}^{\eta,\varphi}(\Omega,\Gamma):
\mbox{\eqref{f10} is satisfied}\bigr\}.
\end{aligned}
\end{equation}
It remains to note that the density of $C^{\infty}(\overline{\Omega})$ in $H^{\varphi\varrho^{2q}}_{A,\eta}(\Omega)$ is a consequence of the dense continuous embedding of $H^{s_1+2q}(\Omega)$ into $H^{\varphi\varrho^{2q}}_{A,\eta}(\Omega)$. This embedding is due to \eqref{f6.7}. The case $\sigma_{1}(\varphi)\geq-1/2$ is examined.

Assume now that $\sigma_{1}(\varphi)<-1/2$. Since $s_0<s_1<-1/2$ in this case, we have the Fredholm bounded operators \eqref{f6.2} due to Proposition~\ref{prop6.1}. Using the same $\alpha$, $r_0$, $r_1$, and  interpolation parameter $\psi$ as in the previous case, we conclude that a restriction of \eqref{f6.2} for $j=0$ is a bounded operator between the spaces \eqref{f6.3}. This operator is Fredholm with kernel $N$ and index $\dim N-\dim N^{+}$ by Proposition~\ref{prop6.1} and \cite[Theorem~1.7]{MikhailetsMurach14}.
According to Theorem~\ref{th5.3} and Remark~\ref{rem5.4}, we have
\begin{equation}\label{f6.9}
\bigl[H^{s_0+2q}_{A,\lambda}(\Omega),
H^{s_1+2q}_{A,\lambda}(\Omega)\bigr]_{\psi}=
H^{\varphi\varrho^{2q}}_{A,\lambda}(\Omega).
\end{equation}
Besides,
\begin{equation*}
\bigl[\mathcal{H}^{\lambda,s_0}(\Omega,\Gamma),
\mathcal{H}^{\lambda,s_1}(\Omega,\Gamma)\bigr]_{\psi}=
\mathcal{H}^{\lambda,\varphi}(\Omega,\Gamma)=
\mathcal{H}^{\eta,\varphi}(\Omega,\Gamma)
\end{equation*}
due to Proposition~\ref{prop5.2}. Hence, the Fredholm bounded operator \eqref{f6.3} acts between the spaces \eqref{f15}, with \eqref{f6.8} holding due to \cite[Theorem~1.7]{MikhailetsMurach14}. Now, the density of $C^{\infty}(\overline{\Omega})$ in $H^{\varphi\varrho^{2q}}_{A,\eta}(\Omega)$ is a consequence of Proposition~\ref{prop6.1} and the dense continuous embedding of  $H^{s_1+2q}_{A,\lambda}(\Omega)$ into $H^{\varphi\varrho^{2q}}_{A,\eta}(\Omega)$. This embedding is due to \eqref{f6.9}. The case $\sigma_{1}(\varphi)<-1/2$ is also examined.
\end{proof}

To prove Lema~\ref{lema1} and other results, we need the scale  $\{H^{r,(2q)}(\Omega):r\in\mathbb{R}\}$ of Hilbert spaces introduced by  Roitberg \cite{Roitberg64}. This scale is applied in the theory of elliptic problems \cite{Berezansky68, KozlovMazyaRossmann97, Roitberg96, Roitberg99}. We will mainly follow \cite[Section~1.10 and Chapter~2]{Roitberg96}.

We first consider the Hilbert space $H^{r,(0)}(\Omega)$ used in the definition of $H^{r,(2q)}(\Omega)$. Let $H^{r,(0)}(\Omega):=H^{r}(\Omega)$ in the $r\geq0$ case. If  $r<0$, then $H^{r,(0)}(\Omega)$ is defined to be the dual of $H^{-r}(\Omega)$ with respect to the inner product in $L_{2}(\Omega)$. Namely, $H^{r,(0)}(\Omega)$, where $r<0$, is the completion of $C^{\infty}(\overline{\Omega})$ with respect to the Hilbert norm
\begin{equation}\label{f6.10}
\|u\|_{r,(0),\Omega}:=\sup\biggl\{\frac{|(u,w)_{\Omega}|}
{\;\quad\|w\|_{-r,\Omega}}:
w\in H^{-r}(\Omega),\,w\neq0\biggr\}.
\end{equation}
Thus, the inner product in $L_{2}(\Omega)$ extends by continuity to a sesquilinear form $(u_1,u_2)_{\Omega}$ defined for arbitrary $u_1\in H^{r,(0)}(\Omega)$ and $u_2\in H^{-r,(0)}(\Omega)$, with $r\in\mathbb{R}$. The norm in $H^{r,(0)}(\Omega)$ is denoted by $\|\cdot\|_{r,(0),\Omega}$ for every $r\in\mathbb{R}$.

Now we can define the Hilbert space $H^{r,(2q)}(\Omega)$. Let
$E_{2q}:=\{1/2,3/2,\ldots,2q-1/2\}$. If $r\in\mathbb{R}\setminus E_{2q}$, then the space $H^{r,(2q)}(\Omega)$ is defined to be the completion of $C^{\infty}(\overline{\Omega})$ with respect to the Hilbert norm
\begin{equation}\label{f6.11}
\|u\|_{r,(2q),\Omega}:=\Biggl(\|u\|_{r,(0),\Omega}^{2}+
\sum_{k=1}^{2q}\;\|(\partial_{\nu}^{k-1}u)\!\upharpoonright\!\Gamma\|
_{r-k+1/2,\Gamma}^{2}\Biggr)^{1/2}.
\end{equation}
(As above, $\partial_\nu$ is the operator of the differentiation along the inward normal to $\Gamma$.) If $r\in E_{2q}$, then we put
\begin{equation*}
H^{r,(2q)}(\Omega):=\bigl[H^{r-\varepsilon,(2q)}(\Omega),
H^{r+\varepsilon,(2q)}(\Omega)\bigr]_{\psi}
\end{equation*}
where $\psi(t)\equiv \sqrt{t}$ and $0<\varepsilon<1$. The right-hand side of this equality does not depend on the choice $\varepsilon$ up to equivalence of norms. The norm in the Hilbert space $H^{r,(2q)}(\Omega)$ is denoted by $\|\cdot\|_{r,(2q),\Omega}$ for every $r\in\mathbb{R}$.

Let $p\in\{0,2q\}$. If $-\infty<r_0<r_1<\infty$, then the identity mapping on $C^{\infty}(\overline{\Omega})$ extends uniquely to a continuous imbedding operator $H^{r_1,(p)}(\Omega)\hookrightarrow H^{r_0,(p)}(\Omega)$. Besides,
\begin{equation}\label{f6.12}
\mbox{if}\;\;r>p-1/2,\;\;\mbox{then}\;\;H^{r,(p)}(\Omega)=H^{r}(\Omega)
\end{equation}
as completions of $C^{\infty}(\overline{\Omega})$ with respect to equivalent norms.

We will use the next result in the proof of Lemma~\ref{lema1}.

\begin{proposition}\label{prop6.2}
Let $\omega\in\mathrm{OR}$, $\sigma_{0}(\omega)>-1/2$, $r\in\mathbb{R}$, and
\begin{equation}\label{f6.13}
r\notin\{-k+1/2:1\leq k\in\mathbb{Z}\}.
\end{equation}
Then the norms
\begin{equation}\label{f6.14}
\|u\|_{r,A,\omega}:=\bigl(\|u\|^{2}_{r,\Omega}+
\|Au\|^{2}_{\omega,\Omega}\bigr)^{1/2}
\end{equation}
and
\begin{equation}\label{f6.15}
\|u\|_{r,(2q),A,\omega}:=\bigl(\|u\|^{2}_{r,(2q),\Omega}+
\|Au\|^{2}_{\omega,\Omega}\bigr)^{1/2}
\end{equation}
are equivalent on the class of all functions $u\in C^{\infty}(\overline{\Omega})$.
\end{proposition}

Recall that \eqref{f6.14} is the norm in $H^{r}_{A,\omega}(\Omega)$. If $r>2q-1/2$, then Proposition~\ref{prop6.2} follows immediately from
\eqref{f6.12}. If $r\leq2q-1/2$, then this proposition is a direct consequence of the isomorphism (4.196) from monograph \cite[Section~4.4.2, Proof of Theorem 4.25]{MikhailetsMurach14}. We put $\sigma:=r-2q$, $L:=A$ and $X^{\sigma}(\Omega):=H^{\omega}(\Omega)$ in this isomorphism and note that the space $H^{\omega}(\Omega)$ satisfies Condition~$\mathrm{I}_{\sigma}$ (used in \cite[Theorem 4.25]{MikhailetsMurach14}) in view of \cite[Theorem~4.26]{MikhailetsMurach14} and the continuous embedding $H^{\omega}(\Omega)\hookrightarrow H^{\lambda}(\Omega)$ for some $\lambda>-1/2$.

\begin{proof}[Proof of Lemma $\ref{lema1}$]
Assume in addition that $r$ satisfies \eqref{f6.13}. Choosing functions $u,w\in C^{\infty}(\overline{\Omega})$ arbitrarily, we get
\begin{align*}
|(u,w)_{\Omega}|&\leq\|u\|_{r,(0),\Omega}\cdot\|w\|_{-r,\Omega}\leq
\|u\|_{r,(2q),A,\omega}\cdot\|w\|_{-r,\Omega}\leq
c_1\|u\|_{r,A,\omega}\cdot\|w\|_{-r,\Omega}\\
&\leq c\,\|u\|_{\alpha,A,\omega}\cdot\|w\|_{-r,\Omega}
\end{align*}
by \eqref{f6.10}, \eqref{f6.11}, \eqref{f6.15}, Proposition~\ref{prop6.2}, and \eqref{f3.3}; here, $c_1$ and $c$ are certain positive numbers that do not depend on $u$ and $w$. If $r$ is not subject to \eqref{f6.13}, then we choose a non-half-integer number $r_1$ such that  $r<r_1<\sigma_0(\alpha)$. As has been proved,
\begin{equation*}
|(u,w)_{\Omega}|\leq c\,\|u\|_{\alpha,A,\omega}\cdot\|w\|_{-r_1,\Omega}\leq
c\,\|u\|_{\alpha,A,\omega}\cdot\|w\|_{-r,\Omega}.
\end{equation*}
The required bound \eqref{lema1-bound} is substantiated.
\end{proof}

\begin{proof}[Proof of Theorem~$\ref{th2}$]
According to Theorem~\ref{th1}, the bounded linear operator \eqref{isom} is a bijection. Hence, it is an isomorphism due to the Banach theorem on inverse operator.
\end{proof}

\begin{proof}[Proof of Theorem~$\ref{th4.6}$]
By the hypotheses of the theorem, we have the inclusion
\begin{equation*}
(g,f)=(A,B)u\in\mathcal{H}^{\eta,\varphi}(\Omega,\Gamma),
\end{equation*}
with $u\in H^{s+2q}_{A,\lambda}(\Omega)\supset H^{\varphi\varrho^{2q}}_{A,\eta}(\Omega)$ for certain  $s<\sigma_0(\varphi)$. Therefore,
\begin{equation*}
(g,f)\in\mathcal{H}^{\eta,\varphi}(\Omega,\Gamma)
\cap(A,B)\bigl(H^{s+2q}_{A,\lambda}(\Omega)\bigr)= (A,B)\bigl(H^{\varphi\varrho^{2q}}_{A,\eta}(\Omega)\bigr),
\end{equation*}
the last equality being due to Theorem~\ref{th1}. Hence, along with the condition $(A,B)u=(g,f)$, the equality $(A,B)v=(g,f)$ holds true for certain $v\in H^{\varphi\rho^{2q}}_{A,\eta}(\Omega)$. Thus, the distribution $w:=u-v\in H^{s+2q}_{A,\lambda}(\Omega)$ satisfies $(A,B)w=0$. Therefore, $w\in N\subset C^{\infty}(\overline{\Omega})$ due to Theorem~\ref{th1}, which gives the required inclusion $u=v+w\in H^{\varphi\rho^{2q}}(\Omega)$.
\end{proof}

\begin{proof}[Proof of Theorem~$\ref{th4.7}$]
We arbitrarily choose a function $\chi\in C^{\infty}(\overline{\Omega})$ such that $\mathrm{supp}\,\chi\subset\Omega_0\cup\nobreak\Gamma_0$ and take a function $\zeta\in C^{\infty}(\overline{\Omega})$ such that $\mathrm{supp}\,\zeta\subset\Omega_0\cup\Gamma_0$ and $\zeta=1$ in a certain neighbourhood $V$ of $\mathrm{supp}\,\chi$ (in the topology of $\overline{\Omega}$, of course). Owing to the hypotheses of the theorem, we have the inclusion
\begin{equation}\label{incl-(f,g)}
\zeta(f,g):=\bigl(\zeta f,(\zeta\!\upharpoonright\!\Gamma)g_1,\ldots, (\zeta\!\upharpoonright\!\Gamma)g_q\bigr)
\in\mathcal{H}^{\eta,\varphi}(\Omega,\Gamma).
\end{equation}
Besides, $u\in H^{s+2q}_{A,\ell}(\Omega)\supset H^{\varphi\varrho^{2q}}_{A,\eta}(\Omega)$ for certain  $s<\sigma_0(\varphi)$ and $\ell\in(-1/2,\sigma_0(\eta))$. We may and will assume that the number $r:=s+2q$ satisfies \eqref{f6.13}.

The space $H^{s+2q}_{A,\ell}(\Omega)$ is not closed with respect to the multiplication by functions from $C^{\infty}(\overline{\Omega})$. This is a reason why we may not deduce this theorem from only Theorem~\ref{th4.6} in a usual manner (see, e.g., \cite[Chapter~III, Section~6, Subsection~11]{Berezansky68} or \cite[Section~4.1.2]{MikhailetsMurach14}). We have to use Roitberg's theorem on local regularity of solutions to the elliptic problem \cite[Theorem~7.2.1]{Roitberg96}. This theorem deals with the solutions of class $H^{s+2q,(2q)}(\Omega)$.

Owing to Proposition~\ref{prop6.2} for $\omega(t)\equiv t^{\ell}$, the identity mapping on $C^{\infty}(\overline{\Omega})$ extends uniquely (by continuity) to a bounded linear operator
\begin{equation}\label{op-O}
O:H^{s+2q}_{A,\ell}(\Omega)\to H^{s+2q,(2q)}(\Omega).
\end{equation}
Let us show that this operator is one-to-one. Suppose that $Ou=0$ for certain $u\in H^{s+2q}_{A,\ell}(\Omega)$. Then there exists a sequence $(u_k)_{k=1}^{\infty}\subset C^{\infty}(\overline{\Omega})$ such that $u_k\to u$ in $H^{s+2q}_{A,\ell}(\Omega)$ and $u_k\to0$ in $H^{s+2q,(2q)}(\Omega)$ as $k\to\infty$. Therefore, $Au_k\to Au$ in $H^{\ell}(\Omega)$ and $Au_k\to0$ in $H^{s,(0)}(\Omega)$, with the latter convergence being due to \cite[Lemma~2.3.1(ii)]{Roitberg96}. This implies $Au=0$ in view of the continuous embedding of the space $H^{\ell}(\Omega)=H^{\ell,(0)}(\Omega)$ in $H^{s,(0)}(\Omega)$ (see \eqref{f6.12}, and take $s<-1/2<\ell$ into account). Hence, $\|u_k\|_{s+2q,(2q),A,\ell}\to0$, which implies $\|u_k\|_{s+2q,A,\ell}\to0$ by Proposition~\ref{prop6.2}. Thus, $u=0$, i.e. the operator \eqref{op-O} is one-to-one. It sets the continuous embedding of $H^{s+2q}_{A,\ell}(\Omega)$ in $H^{s+2q,(2q)}(\Omega)$. We may therefore consider the distribution $u\in H^{s+2q}_{A,\ell}(\Omega)$ as an element of the Roitberg's space $H^{s+2q,(2q)}(\Omega)$.

According to \eqref{incl-(f,g)} and Theorem~\ref{th2}, there exists a distribution
\begin{equation}\label{v-in}
v\in H^{\varphi\rho^{2q}}_{A,\eta}(\Omega)
\end{equation}
such that
\begin{equation*}
(A,B)v=\mathcal{P}^{+}(\zeta(f,g))\in
\mathcal{H}^{\eta,\varphi}(\Omega,\Gamma).
\end{equation*}
Putting
\begin{equation*}
w:=u-v\in H^{s+2q}_{A,\ell}(\Omega)\subset H^{s+2q,(2q)}(\Omega),
\end{equation*}
we see that
\begin{equation}\label{(A,B)w}
(A,B)w=(f,g)-\mathcal{P}^{+}(\zeta(f,g))=:F\in H^{s,(0)}(\Omega)\oplus
\bigoplus_{j=1}^{q}H^{s+2q-m_j-1/2}(\Gamma)
\end{equation}
because $(f,g)=(A,B)u$ and $u\in H^{s+2q}_{A,\ell}(\Omega)$ and because the space $\mathcal{H}^{\eta,\varphi}(\Omega,\Gamma)$ is narrower than the orthogonal sum in \eqref{(A,B)w}. Besides,
\begin{align*}
\chi_{1}F=&
\chi_{1}\bigl((f,g)-\zeta(f,g)+(I-\mathcal{P}^{+})(\zeta(f,g))\bigr)\\
=&\chi_{1}(I-\mathcal{P}^{+})(\zeta(f,g))\in C^{\infty}(\overline{\Omega})\times\{0\}^{q}
\end{align*}
for every function $\chi_{1}\in C^{\infty}(\overline{\Omega})$ subject to $\mathrm{supp}\,\chi_{1}\subset V$ (recall that $\zeta=1$ on $V$). Hence, \begin{equation}\label{chi_{1}w}
\chi_{1}w\in\bigcap_{r>s+2q}H^{r,(2q)}(\Omega)=
C^{\infty}(\overline{\Omega})
\end{equation}
for every above-mentioned function $\chi_{1}$. This conclusion is due to \cite[Theorem~7.2.1]{Roitberg96} (or \cite[Theorem~7.2.2]{Roitberg96} if $\Gamma_0=\emptyset$), with the equation in \eqref{chi_{1}w} holding in view of \eqref{f6.12}. Taking $\chi_{1}:=\chi$, we obtain
\begin{equation*}
\chi u=\chi v+\chi w\in H^{\varphi\rho^{2q}}(\Omega)
\end{equation*}
by \eqref{v-in} and \eqref{chi_{1}w}. Thus, $u\in H^{\varphi\rho^{2q}}_{\mathrm{loc}}(\Omega_0,\Gamma_0)$ due to the arbitrariness of our choice of $\chi$.
\end{proof}

\begin{proof}[Proof of Theorem~$\ref{th4.9}$]
We choose a sufficiently small number $\varepsilon>0$, put $U_{\varepsilon}:=\{x\in U:\mathrm{dist}(x,\partial U)>\varepsilon\}$, $\Omega_{\varepsilon}:=\Omega\cap U_{\varepsilon}$, and $\Gamma_{\varepsilon}:=\Gamma\cap U_{\varepsilon}$, and consider a function $\chi_{\varepsilon}\in C^\infty(\overline\Omega)$ such that $\mbox{supp}\,\chi_{\varepsilon}\subset\Omega_0\cup\Gamma_0$ and $\chi_{\varepsilon}=1$ on $\Omega_{\varepsilon}\cup\Gamma_{\varepsilon}$. Owing to Theorem~\ref{th4.7}, the inclusion $\chi_{\varepsilon}u\in H^{\varphi\varrho^{2q}}(\Omega)$ holds true. Hence, there exists a distribution $w_{\varepsilon}\in H^{\varphi\varrho^{2q}}(\mathbb{R}^{n})$ such that $w_{\varepsilon}=\chi_{\varepsilon}u=u$ on $\Omega_{\varepsilon}$. By \eqref{Hermander-embedding}, condition \eqref{int-cond} implies $w_{\varepsilon}\in C^{p}(\mathbb{R}^{n})$. Therefore,
\begin{equation*}
(u,v)_{\Omega}=
\int\limits_{\Omega_{\varepsilon}}
w_{\varepsilon}(x)\overline{v(x)}dx=
\int\limits_{\Omega_{0}}u_{0}(x)\overline{v(x)}dx
\end{equation*}
for every $v\in C^{\infty}_{0}(\Omega)$ subject to $\mathrm{supp}\,v\subset\Omega_{\varepsilon}$. Here, the function $u_{0}\in C^{p}(\Omega_{0}\cup\Gamma_{0})$ is defined by the formula $u_{0}:=w_{\varepsilon}$ on $\Omega_{\varepsilon}\cup\Gamma_{\varepsilon}$ whenever $0<\varepsilon\ll1$. This function is well defined because $0<\delta<\varepsilon$ implies that $w_{\delta}=w_{\varepsilon}$ on $\Omega_{\varepsilon}\cup\Gamma_{\varepsilon}$. Thus, $u$ satisfies \eqref{def-C^p}, i.e. $u\in C^{p}(\Omega_{0}\cup\Gamma_{0})$.
\end{proof}

Let us now substantiate Remark~\ref{rem4.10}. Namely, we suppose that the implication \eqref{implication} holds true for a certain integer $p\geq0$ and will prove that $\varphi$ satisfies \eqref{int-cond}. Choosing a distribution $u\in H^{\varphi\rho^{2q}}_{A,\eta}(\Omega)$ arbitrarily, we define the right-hand sides
of the problem \eqref{f1}, \eqref{f2}. They satisfy the inclusion $(f,g)\in\mathcal{H}^{\varphi,\eta}(\Omega,\Gamma)$ and, hence, the hypotheses of Theorem~$\ref{th4.7}$. Therefore,
\begin{equation}\label{f6.21}
u\in H^{\varphi\rho^{2q}}_{A,\eta}(\Omega)\;\Longrightarrow\;
u\in C^{p}(\Omega_{0}\cup\Gamma_{0})
\end{equation}
by \eqref{implication}. We now choose a distribution $g_1\in H^{\varphi\rho^{2q-1/2}}(\Gamma)$ arbitrarily, put $g_j:=0$ whenever $2\leq j\leq q$, and consider the regular elliptic problem that consists of the equation \eqref{f1} and boundary conditions
\begin{equation}\label{f6.22}
\partial^{j-1}_{\nu}u=g_j\quad\mbox{on}\;\Gamma,\quad j=1,...,q.
\end{equation}
Here, $\partial_{\nu}$ is the operator of differentiation along the inner normal $\nu$ to $\Gamma=\partial\Omega$. According to \eqref{sum2b}, there exists a function $f\in N^{+}\subset C^{\infty}(\overline{\Omega})$ such that $\mathcal{P}^{+}(0,g)=(f,g)$.
We take this function to be the right-hand side of \eqref{f1}. Owing to Theorem~\ref{th2}, the elliptic problem \eqref{f1}, \eqref{f6.22} has a solution $u\in H^{\varphi\rho^{2q}}_{A,\eta}(\Omega)$. This solution belongs to $C^{p}(\Omega_{0}\cup\Gamma_{0})$ due to \eqref{f6.21}. Hence, the restriction of every distribution $g_1\in H^{\varphi\rho^{2q-1/2}}(\Gamma)$ to $\Gamma_{0}$ pertains to
$C^{p}(\Gamma_{0})$. Passing to local coordinates on $\Gamma$, we deduce plainly from this fact that
\begin{equation}\label{f6.23}
\bigl\{w\in H^{\varphi\rho^{2q-1/2}}(\mathbb{R}^{n-1}): \mathrm{supp}\,w\subset V\bigr\}\subset C^p(\mathbb{R}^{n-1})
\end{equation}
for a certain open subset $V\neq\emptyset$ of $\mathbb{R}^{n-1}$. The inclusion \eqref{f6.23} implies \eqref{int-cond} due to \eqref{Hermander-embedding}, which substantiates Remark~\ref{rem4.10}.

\begin{proof}[Proof of Theorem~$\ref{th4.12}$]
According to Peetre's lemma \cite[Lemma~3]{Peetre61}, this theorem is a consequence of the facts that the operator \eqref{f15} has finite-dimensional kernel and closed range by Theorem~\ref{th1} and that the embedding $H^{\varphi\varrho^{2q}}_{A,\eta}(\Omega)\hookrightarrow H^{\varphi\varrho^{2q-\ell}}(\Omega)$ is compact (in fact, the continuity of this embedding is enough). However, it is not difficult to prove this theorem not referring to the mentioned lemma. Namely, using the decomposition \eqref{sum1}, we represent an arbitrary distribution $u\in H^{\varphi\rho^{2q}}_{A,\eta}(\Omega)$ in the form $u=u_{0}+u_{1}$ with $u_{0}:=(1-P)u\in N$ and $u_{1}:=Pu$. Owing to Theorem~\ref{th2}, we get
\begin{equation}\label{f6.24}
\|u_1\|_{\varphi\rho^{2q},\Omega}\leq
\|u_1\|_{\varphi\rho^{2q},A,\eta}\leq
c_{1}\|(A,B)u_1\|_{\eta,\varphi,\Omega,\Gamma}=
c_{1}\|(f,g)\|_{\eta,\varphi,\Omega,\Gamma},
\end{equation}
with $c_1$ being the norm of the inverse operator to the isomorphism \eqref{isom}. Since the space $N$ is finite-dimensional, all norms are equivalent on $N$, specifically, the norms in $H^{\varphi\rho^{2q}}(\Omega)$ and $H^{\varphi\rho^{2q-\ell}}(\Omega)$. It follows hence from $u_{0}\in N$ and \eqref{f6.24} that
\begin{align*}
\|u_0\|_{\varphi\rho^{2q},\Omega}&\leq c_{0}\|u_0\|_{\varphi\rho^{2q-\ell},\Omega}\leq
c_{0}\|u\|_{\varphi\rho^{2q-\ell},\Omega}+
c_{0}\|u_1\|_{\varphi\rho^{2q-\ell},\Omega}\\
&\leq c_{0}\|u\|_{\varphi\rho^{2q-\ell},\Omega}+
c_{0}\|u_1\|_{\varphi\rho^{2q},\Omega}\\
&\leq c_{0}\|u\|_{\varphi\rho^{2q-\ell},\Omega}+
c_{0}c_{1}\|(f,g)\|_{\eta,\varphi,\Omega,\Gamma};
\end{align*}
here, $c_0$ is a positive number that does not depend on $u$. This together with \eqref{f6.24} yields the required estimate \eqref{f4.20}. It remains to remark that $u$ from Theorem~\ref{th4.12} belongs to $H^{\varphi\rho^{2q}}_{A,\eta}(\Omega)$ by Theorem~\ref{th4.6}.
\end{proof}

\begin{proof}[Proof of Theorem~$\ref{th4.13}$]
Note previously that we may not deduce this theorem from Theorem~\ref{th4.12} by a usual reasoning (compare, e.g., with \cite[Section~4.1.2, pp. 170--172]{MikhailetsMurach14} or \cite[Section~7.2, p.~216]{Roitberg96}) because the right-hand side of \eqref{f4.20} contains the norm $\|f\|_{\eta,\Omega}$ instead of the norm $\|f\|_{\varphi,\Omega}$, which is necessary to perform this reasoning. Besides, the hypothesis $Au\in H^{-1/2+}(\Omega)$ does not imply that $A(\chi u)\in H^{-1/2+}(\Omega)$. However, we will need the latter inclusion if we use Theorem~\ref{th4.12} for $\chi u$ instead of $u$ according to the usual reasoning. Thus, we have to choose another way to prove Theorem~\ref{th4.13}. This way involves the Roitberg spaces $H^{r,(2q)}(\Omega)$, with $r\in\mathbb{R}$, used in our previous proofs. We divide our reasoning into four steps.

\emph{Step~$1$.} According to \cite[Theorem 4.1.1]{Roitberg96}, the mapping \eqref{mapping} extends uniquely (by continuity) to a Fredholm bounded operator
\begin{equation}\label{f6.25}
(A,B):H^{s+2q,(2q)}(\Omega)\rightarrow H^{s,(0)}(\Omega)\oplus
\bigoplus_{j=1}^{q}H^{s+2q-m_j-1/2}(\Gamma)
\end{equation}
for every $s\in\mathbb{R}$. The kernel and index of this operator do not depend on $s$ and are equal to $N$ and $\dim N-\dim N^{+}$ resp. (Observe in view of \eqref{f6.12} that \eqref{f6.25} coincides with \eqref{f9} whenever $\varphi(t)\equiv t^{s}$ and $s>-1/2$.) We interpolate the spaces involved in \eqref{f6.25} and use the interpolation parameter $\psi$ defined by formula \eqref{f5.1} in which $r_{0}:=s_{0}$, $r_{1}:=s_{1}$, and
$\alpha:=\varphi$. Owing to Proposition~\ref{prop5.2}, the equality
\begin{equation}\label{f6.26}
[H^{s_{0}+r}(G),H^{s_{1}+r}(G)]_{\psi}=H^{\varphi\rho^{r}}(G)
\quad\mbox{for every}\quad r\in\mathbb{R}
\end{equation}
holds true up to equivalence of norms, with $G\in\{\mathbb{R}^{n},\Omega,\Gamma\}$. Given $r\in\mathbb{R}$, we define the Hilbert spaces
\begin{equation*}
X_{r}:=[H^{s_{0}+r,(2q)}(\Omega),H^{s_{1}+r,(2q)}(\Omega)]_{\psi}
\end{equation*}
and
\begin{equation*}
Y_{r}:=[H^{s_{0}+r,(0)}(\Omega),H^{s_{1}+r,(0)}(\Omega)]_{\psi}.
\end{equation*}
Consider the Fredholm operators \eqref{f6.25} for each $s\in\{s_{0}+r,s_{1}+r\}$. Interpolating them with the function parameter $\psi$, we conclude by \cite[Theorem~1.7]{MikhailetsMurach14} that the restriction of the mapping \eqref{f6.25}, where $s=s_{0}+r$, on the space $X_{2q+r}$ is a Fredholm bounded operator
\begin{equation}\label{f6.27}
(A,B):X_{2q+r}\to Y_{r}\oplus Z_{r}.
\end{equation}
Here, the Hilbert space
\begin{equation*}
Z_{r}:=\bigoplus_{j=1}^{q}H^{\varphi\varrho^{2q+r-m_j-1/2}}(\Gamma)
\end{equation*}
equals
\begin{equation*}
\bigoplus_{j=1}^{q}
[H^{s_{0}+r+2q-m_j-1/2}(\Gamma),H^{s_{1}+r+2q-m_j-1/2}(\Gamma)]_{\psi}
\end{equation*}
up to equivalence of norms due to \eqref{f6.26}.

Let $0\leq k\in\mathbb{Z}$, and let $\zeta_{1}\in C^{\infty}(\overline{\Omega})$ satisfy $\zeta_{1}=1$ in a neighbourhood of $\mathrm{supp}\,\chi$. We will prove by induction in $k$ that
\begin{equation}\label{f6.28}
\|\chi u\|_{X_{2q}}\leq c_{0}\bigl(\|\zeta_{1}Au\|_{Y_{0}}+\|\zeta_{1}Bu\|_{Z_{0}}+
\|\zeta_{1}u\|_{X_{2q-k}}\bigr)
\end{equation}
for every $u\in C^{\infty}(\overline{\Omega})$ with a certain number $c_{0}>0$ that does not depend on $u$. (We use the standard notation for the norms in the spaces $X_{r}$, $Y_{r}$, and $Z_{r}$.) Let us $c_{1}$, $c_{2}$,... denote some positive numbers that are independent of $u$.

If $k=0$, then \eqref{f6.28} follows from
\begin{equation*}
\|\chi u\|_{X_{2q}}=\|\chi\zeta_{1}u\|_{X_{2q}}\leq c_{1}\|\zeta_{1}u\|_{X_{2q}}.
\end{equation*}
The latter inequality is true because the operator of the multiplication by a function from $C^{\infty}(\overline{\Omega})$ is bounded on every space $H^{r,(2q)}(\Omega)$ (as well as on $H^{r,(0)}(\Omega)$); see \cite[Corollary 2.3.1]{Roitberg96}. Assume now that \eqref{f6.28} holds true for a certain integer $k=p\geq0$, and prove \eqref{f6.28} in the case of $k=p+1$.

Consider a function $\zeta_{0}\in C^{\infty}(\overline{\Omega})$ such that  $\zeta_{0}=1$ in a neighbourhood of $\mathrm{supp}\,\chi$ and that $\zeta_{1}=1$ in a neighbourhood of $\mathrm{supp}\,\zeta_{0}$. By the inductive assumption \eqref{f6.28}, we have
\begin{equation}\label{f6.29}
\|\chi u\|_{X_{2q}}\leq c_{0}\bigl(\|\zeta_{0}Au\|_{Y_{0}}+\|\zeta_{0}Bu\|_{Z_{0}}+
\|\zeta_{0}u\|_{X_{2q-p}}\bigr).
\end{equation}
Since the bounded operator \eqref{f6.27}, where $r:=-p$, is Fredholm, we have the estimate
\begin{equation}\label{f6.30}
\|\zeta_{0}u\|_{X_{2q-p}}\leq c_{2}\bigl(\|A(\zeta_{0}u)\|_{Y_{-p}}+\|B(\zeta_{0}u)\|_{Z_{-p}}+
\|\zeta_{0}u\|_{X_{2q-p-1}}\bigr)
\end{equation}
due to the above-mentioned lemma by Peetre \cite[Lemma~3]{Peetre61}. Interchanging the PDO $A$ with the operator of multiplication by $\zeta_{0}$, we get
\begin{equation*}
A(\zeta_{0}u)=A(\zeta_{0}\zeta_{1}u)=\zeta_{0}A(\zeta_{1}u)+A'(\zeta_{1}u)=
\zeta_{0}Au+A'(\zeta_{1}u).
\end{equation*}
Here, $A'$ is a certain linear PDO on $\overline{\Omega}$ whose  coefficients belong to $C^{\infty}(\overline{\Omega})$ and whose order $\mathrm{ord}\,A'\leq2q-1$. Analogously,
\begin{equation*}
B(\zeta_{0}u)=\zeta_{0}Bu+B'(\zeta_{1}u);
\end{equation*}
here, $B':=(B_{1}',\ldots,B_{q}')$ where each $B_{j}'$ is a certain linear boundary PDO on $\Gamma$ with coefficients of class $C^{\infty}(\Gamma)$ and of order $\mathrm{ord}\,B_{j}'\leq m_{j}-1$. By \cite[Lemma~2.3.1]{Roitberg96} and the interpolation, the PDOs $A'$ and $B_{j}'$ act continuously between the following spaces:
\begin{align*}
A'&:X_{2q-p-1}=
[H^{s_{0}+2q-p-1,(2q)}(\Omega),H^{s_{1}+2q-p-1,(2q)}(\Omega)]_{\psi}\\
&\to[H^{s_{0}-p,(2q)}(\Omega),H^{s_{1}-p,(2q)}(\Omega)]_{\psi}=Y_{-p}
\end{align*}
and
\begin{equation*}
B_{j}':X_{2q-p-1}\to
[H^{s_{0}+2q-p-m_{j}-1/2}(\Gamma),H^{s_{1}+2q-p-m_{j}-1/2}(\Gamma)]_{\psi}
=H^{\varphi\varrho^{2q-p-m_j-1/2}}(\Gamma)
\end{equation*}
in view of \eqref{f6.26}. Thus, it follows from \eqref{f6.30} that
\begin{align*}
\|\zeta_{0}u\|_{X_{2q-p}}&\leq c_{2}\bigl(\|\zeta_{0}Au\|_{Y_{-p}}+
\|A'(\zeta_{1}u)\|_{Y_{-p}}+\|\zeta_{0}Bu\|_{Z_{-p}}+
\|B'(\zeta_{1}u)\|_{Z_{-p}}+\|\zeta_{0}u\|_{X_{2q-p-1}}\bigr)\\
&\leq c_{3}\bigl(\|\zeta_{0}Au\|_{Y_{-p}}+
\|\zeta_{1}u\|_{X_{2q-p-1}}+\|\zeta_{0}Bu\|_{Z_{-p}}+
\|\zeta_{1}u\|_{X_{2q-p-1}}+\|\zeta_{0}u\|_{X_{2q-p-1}}\bigr)\\
&=c_{3}\bigl(\|\zeta_{0}\zeta_{1}Au\|_{Y_{-p}}+
\|\zeta_{0}\zeta_{1}Bu\|_{Z_{-p}}+2\|\zeta_{1}u\|_{X_{2q-p-1}}+
\|\zeta_{0}\zeta_{1}u\|_{X_{2q-p-1}}\bigr)\\
&\leq c_{4}\bigl(\|\zeta_{1}Au\|_{Y_{-p}}+
\|\zeta_{1}Bu\|_{Z_{-p}}+\|\zeta_{1}u\|_{X_{2q-p-1}}\bigr).
\end{align*}
Applying this estimate and
\begin{equation*}
\|\zeta_{0}Au\|_{Y_{0}}+\|\zeta_{0}Bu\|_{Z_{0}}=
\|\zeta_{0}\zeta_{1}Au\|_{Y_{0}}+\|\zeta_{0}\zeta_{1}Bu\|_{Z_{0}}
\leq c_{5}\bigl(\|\zeta_{1}Au\|_{Y_{0}}+\|\zeta_{1}Bu\|_{Z_{0}}\bigr)
\end{equation*}
to \eqref{f6.29}, we obtain
\begin{align*}
\|\chi u\|_{X_{2q}}&\leq c_{0}c_{5}\bigl(\|\zeta_{1}Au\|_{Y_{0}}+\|\zeta_{1}Bu\|_{Z_{0}}\bigr)+
c_{0}c_{4}\bigl(\|\zeta_{1}Au\|_{Y_{-p}}+\|\zeta_{1}Bu\|_{Z_{-p}}+
\|\zeta_{1}u\|_{X_{2q-p-1}}\bigr)\\
&\leq c_{6}\bigl(\|\zeta_{1}Au\|_{Y_{0}}+\|\zeta_{1}Bu\|_{Z_{0}}+
\|\zeta_{1}u\|_{X_{2q-p-1}}\bigr).
\end{align*}
Here, we use the continuous embeddings $Y_{0}\hookrightarrow Y_{-p}$ and
$Z_{0}\hookrightarrow Z_{-p}$, which hold true due to the definitions of the involved spaces via the interpolation. Thus, the estimate \eqref{f6.28} is proved for $k=p+1$ under the assumption that this estimate is valid for $k=p$. Hence, we have proved the estimate for every integer $k\geq0$.

\emph{Step~$2$.} We will prove some relations between norms involved in the obtained inequality \eqref{f6.28} and the required estimate \eqref{f4.21}. First let us show that
\begin{equation}\label{f6.31}
\|v\|_{\varphi\varrho^{2q},\Omega}\leq \widetilde{c}_{1}\|v\|_{X_{2q}}
\quad\mbox{for every}\quad v\in C^{\infty}(\overline{\Omega})
\end{equation}
with some number $\widetilde{c}_{1}>0$ that does not depend on $v$.
Given $v\in C^{\infty}(\overline{\Omega})$, we put $(\mathcal{O}v)(x):=v(x)$ if $x\in\overline{\Omega}$ and put $(\mathcal{O}v)(x):=0$ if $x\in\mathbb{R}^{n}\setminus\overline{\Omega}$.
As is known (see, e.g., \cite[Theorem 4.8.1]{Triebel95}), the mapping $v\mapsto\mathcal{O}v$, where $v\in C^{\infty}(\overline{\Omega})$, extends uniquely (by continuity) to an isometric isomorphism between $H^{r,(0)}(\Omega)$ and the subspace $\{w\in H^{r}(\mathbb{R}^{n}):\mathrm{supp}\,w\subseteq\overline{\Omega}\}$ of $H^{r}(\mathbb{R}^{n})$ provided that $r\leq0$. According to the definition of $H^{r,(2q)}(\Omega)$, we therefore get
\begin{equation*}
\|v\|_{r,\Omega}\leq\|\mathcal{O}v\|_{r,\mathbb{R}^{n}}=\|v\|_{r,(0),\Omega}
\leq\|v\|_{r,(2q),\Omega}\quad\mbox{whenever}\quad r<0.
\end{equation*}
Besides,
\begin{equation*}
\|v\|_{r,\Omega}=\|v\|_{r,(0),\Omega}\leq\|v\|_{r,(2q),\Omega}
\quad\mbox{whenever}\quad r\in[0,\infty)\setminus E_{2q}.
\end{equation*}
Hence, the identity mapping on $C^{\infty}(\overline{\Omega})$ extends uniquely (by continuity) to a bounded linear operator
\begin{equation*}
I_{2q}:H^{r,(2q)}(\Omega)\to H^{r}(\Omega)\quad\mbox{for every}\quad
r\in\mathbb{R}
\end{equation*}
(the $r\in E_{2q}$ case is treated with the help of the interpolation). Considering this operator for $r\in\{s_{0}+2q,s_{1}+2q\}$ and then interpolating with the function parameter $\psi$, we conclude by \eqref{f6.26} that the above-mentioned identity mapping extends uniquely to a bounded linear operator
\begin{equation*}
I_{2q}:X_{2q}=[H^{s_{0}+2q,(2q)}(\Omega),H^{s_{1}+2q,(2q)}(\Omega)]_{\psi}
\to[H^{s_{0}+2q}(\Omega),H^{s_{1}+2q}(\Omega)]_{\psi}=
H^{\varphi\varrho^{2q}}(\Omega).
\end{equation*}
This yields the required inequality \eqref{f6.31}.

Let us now prove that
\begin{equation}\label{f6.33}
\|v\|_{Y_{0}}\leq\widetilde{c}_{2}\|v\|_{\eta,\Omega}
\quad\mbox{for every}\quad v\in C^{\infty}(\overline{\Omega})
\end{equation}
with some number $\widetilde{c}_{2}>0$ that does not depend on $v$. If $\sigma_{1}(\varphi)\geq-1/2$, then $\eta(t)\equiv t^{\lambda}\psi(t^{s_1-\lambda})$ due to \eqref{f6.7b} and then
\begin{align*}
H^{\eta}(\Omega)&=[H^{\lambda}(\Omega),H^{s_1}(\Omega)]_{\psi}=
[H^{\lambda,(0)}(\Omega),H^{s_1,(0)}(\Omega)]_{\psi}\\
&\hookrightarrow[H^{s_0,(0)}(\Omega),H^{s_1,(0)}(\Omega)]_{\psi}=Y_{0}
\end{align*}
due to Proposition \ref{prop5.2}, formula \eqref{f6.12}, and the inequalities $s_0<-1/2<\lambda<s_1$. If $\sigma_{1}(\varphi)<-1/2$, then
\begin{equation*}
H^{\eta}(\Omega)=H^{\lambda}(\Omega)=H^{\lambda,(0)}(\Omega)\hookrightarrow
H^{s_1,(0)}(\Omega)\hookrightarrow Y_{0}
\end{equation*}
due to \eqref{f6.12} and $s_1<-1/2<\lambda$. Note that the written equalities of spaces hold true up to equivalence of norms. Thus, we have the continuous embedding $H^{\eta}(\Omega)\hookrightarrow Y_{0}$ in both cases, which gives \eqref{f6.33}.

Choose $k\in\mathbb{Z}$ such that $k\geq s_1+2q$ and $k\geq s_1-s_0+\ell$, and put $s:=[s_1]+2q-k\leq0$ (as usual, $[s_1]$ stands for the integral part of $s_1$). Then we have the continuous embeddings
\begin{equation}\label{f6.34}
H^{s,(2q)}(\Omega)\hookrightarrow [H^{s_0+2q-k,(2q)}(\Omega),H^{s_1+2q-k,(2q)}(\Omega)]_{\psi}=X_{2q-k}
\end{equation}
and
\begin{equation}\label{f6.35}
H^{\varphi\varrho^{2q-\ell}}(\Omega)\hookrightarrow H^{s_0+2q-\ell}(\Omega)\hookrightarrow H^{s}(\Omega)
\end{equation}
in view of \eqref{f6.26}. Applying \eqref{f6.31}--\eqref{f6.34} to \eqref{f6.28}, we conclude that
\begin{equation}\label{f6.36}
\|\chi u\|_{\varphi\varrho^{2q},\Omega}\leq\widetilde{c}\,
\bigl(\|\zeta_{1}(A,B)u\|_{\eta,\varphi,\Omega,\Gamma}+
\|\zeta_{1}u\|_{s,(2q),\Omega}\bigr)
\end{equation}
for every $u\in C^{\infty}(\overline{\Omega})$, with the number $\widetilde{c}>0$ being independent of~$u$.

\emph{Step~$3$.} Assume on this step that $u\in C^{\infty}(\overline{\Omega})$, and deduce the required estimate \eqref{f4.21} from \eqref{f6.36} under this assumption. Let $V$ be an open set from the topology on $\overline{\Omega}$ such that $\mathrm{supp}\,\chi\subset V$ and that $\zeta=1$ on $\overline{V}$. We may and do choose $V$ so that $V_0:=V\cap\Omega$ is an open domain in $\mathbb{R}^{n}$ with infinitely smooth boundary. Then the Roitberg space $H^{s,(2q)}(V_0)$ is well defined on $V_0$, with $\|\cdot\|_{s,(2q),V_0}$ denoting the norm in this space; recall that $s:=[s_1]+2q-k\leq0$. Let $v$ be the restriction of $u$ to $\overline{V}$; thus, $v\in C^{\infty}(\overline{V})$.

Let a function $\zeta_{1}\in C^{\infty}(\overline{\Omega})$ satisfy the conditions $\mathrm{supp}\,\zeta_{1}\subset V$ and $\zeta_{1}=1$ in a neighbourhood of $\mathrm{supp}\,\chi$. We then have the equivalence of norms
\begin{equation}\label{f6.37}
\|\zeta_{1}u\|_{s,(2q),\Omega}\asymp\|\zeta_{1}v\|_{s,(2q),V_0}
\quad\mbox{with respect to}\quad u\in C^{\infty}(\overline{\Omega}).
\end{equation}
Indeed, according to \cite[Theorem 6.1.1]{Roitberg96}, we get
\begin{align*}
\|\zeta_{1}u\|_{s,(2q),\Omega}&\asymp
\|\zeta_{1}u\|_{s,(0),\Omega}+\|A(\zeta_{1}u)\|_{s-2q,(0),\Omega}
=\|\mathcal{O}(\zeta_{1}u)\|_{s,\mathbb{R}^{n}}+
\|\mathcal{O}A(\zeta_{1}u)\|_{s-2q,\mathbb{R}^{n}}\\
&=\|\zeta_{1}v\|_{s,(0),V_0}+\|A(\zeta_{1}v)\|_{s-2q,(0),V_0}\asymp
\|\zeta_{1}v\|_{s,(2q),V_0}
\end{align*}
with respect to $u$. Recall that $\mathcal{O}(\zeta_{1}u)$ and $\mathcal{O}A(\zeta_{1}u)$ are extensions of the functions $\zeta_{1}v$ and $A(\zeta_{1}v)$, resp., over $\mathbb{R}^{n}$ with zero, these functions being considered on $\overline{V}$ as well as on $\overline{\Omega}$.

Owing to Proposition~\ref{prop6.2}, we have the equivalence of norms
\begin{equation*}
\|v\|_{s,(2q),V_0}+\|Av\|_{\eta,V_0}\asymp
\|v\|_{s,V_0}+\|Av\|_{\eta,V_0}\quad\mbox{with respect to}\quad v\in C^{\infty}(\overline{\Omega}).
\end{equation*}
Combining it with \eqref{f6.37}, we get
\begin{align*}
\|\zeta_{1}u\|_{s,(2q),\Omega}&\asymp\|\zeta_{1}v\|_{s,(2q),V_0}\leq
c_{7}\|v\|_{s,(2q),V_0}\leq c_{7}\bigl(\|v\|_{s,(2q),V_0}+\|Av\|_{\eta,V_0}\bigr)\\
&\asymp\|v\|_{s,V_0}+\|Av\|_{\eta,V_0}=
\|\zeta v\|_{s,V_0}+\|\zeta Av\|_{\eta,V_0}\leq
\|\zeta u\|_{s,\Omega}+\|\zeta Au\|_{\eta,\Omega}\\
&\leq c_{8}\|\zeta u\|_{\varphi\varrho^{2q-\ell},\Omega}+\|\zeta Au\|_{\eta,\Omega}
\end{align*}
in view of \eqref{f6.35}. Thus, there exists a number $\widetilde{c}_{3}>0$ such that
\begin{equation*}
\|\zeta_{1}u\|_{s,(2q),\Omega}\leq\widetilde{c}_{3}
\bigl(\|\zeta u\|_{\varphi\varrho^{2q-\ell},\Omega}+\|\zeta Au\|_{\eta,\Omega}\bigr)
\end{equation*}
for every $u\in C^{\infty}(\overline{\Omega})$. Substituting this inequality into \eqref{f6.36}, we obtain
\begin{align*}
\|\chi u\|_{\varphi\varrho^{2q},\Omega}&\leq\widetilde{c}\:
\|\zeta_{1}\zeta(A,B)u\|_{\eta,\varphi,\Omega,\Gamma}+
\widetilde{c}\:\widetilde{c}_{3}\bigl(
\|\zeta u\|_{\varphi\varrho^{2q-\ell},\Omega}+\|\zeta Au\|_{\eta,\Omega}
\bigr)\\
&\leq c\,\bigl(\|\zeta(A,B)u\|_{\eta,\varphi,\Omega,\Gamma}+
\|\zeta u\|_{\varphi\rho^{2q-\ell},\Omega}\bigr),
\end{align*}
which gives \eqref{f4.21} under the assumption that $u\in C^{\infty}(\overline{\Omega})$.

\emph{Step~$4$.} Recall we must prove that the estimate
\eqref{f4.21} holds true if a distribution $u\in\mathcal{S}'(\Omega)$ satisfies the hypotheses of Theorem~$\ref{th4.7}$. Let us deduce this estimate from the $u\in C^{\infty}(\overline{\Omega})$ case investigated on the previous step. Let $V$ be an open set from the topology on $\overline{\Omega}$ such that $\overline{V}\subset\Omega_{0}\cup\Gamma_{0}$ and ($\mathrm{supp}\,\chi\subset$) $\mathrm{supp}\,\zeta\subset V$ and that $V_0:=V\cap\Omega$ is an open domain in $\mathbb{R}^{n}$ with infinitely smooth boundary $\partial V_{0}$.

Consider an arbitrary distribution $u\in\mathcal{S}'(\Omega)$ that satisfies the hypotheses of Theorem~$\ref{th4.7}$. Then $v:=u\!\upharpoonright\!V_0\in H^{\varphi\varrho^{2q}}_{A,\eta}(V_0)$. Indeed, let a function $\zeta_{1}\in C^{\infty}(\overline{\Omega})$ satisfy the conditions $\mathrm{supp}\,\zeta_{1}\subset\Omega_{0}\cup\Gamma_{0}$ and $\zeta_{1}=1$ on $\overline{V}$; then $\zeta_{1}Au=\zeta_{1}f\in H^{\eta}(\Omega)$ and $\zeta_{1}u\in H^{\varphi\varrho^{2q}}(\Omega)$ due to the hypothesis \eqref{f4.14} and the conclusion of Theorem~$\ref{th4.7}$, respectively; hence, $v\in H^{\varphi\varrho^{2q}}(V_0)$ and $Av\in H^{\eta}(V_0)$, i.e. $v\in H^{\varphi\varrho^{2q}}_{A,\eta}(V_0)$.

Since the set $C^{\infty}(\overline{V})$ is dense in $H^{\varphi\varrho^{2q}}_{A,\eta}(V_0)$ by Theorem~\ref{th1}, there exists a sequence $(u_{k})_{k=1}^{\infty}\subset C^{\infty}(\overline{\Omega})$ such that $v_{k}:=u_{k}\!\upharpoonright\!\overline{V}\to v$ in $H^{\varphi\varrho^{2q}}(V_0)$ and $Av_{k}\to Av$ in $H^{\eta}(V_0)$ as $k\to\infty$. Then
\begin{equation}\label{f6.38}
\zeta u_{k}\to\zeta u\quad\mbox{in}\;\;
H^{\varphi\varrho^{2q}}_{A,\eta}(\Omega)
\end{equation}
and
\begin{equation}\label{f6.39}
\zeta(Au_{k})\to\zeta(Au)\quad\mbox{in}\;\;H^{\eta}(\Omega)
\end{equation}
as $k\to\infty$. Indeed, choose a number $r\gg1$ such that the numbers
$2q+\sigma_{0}(\varphi)$, $2q+\sigma_{1}(\varphi)$, $\sigma_{0}(\eta)$, and $\sigma_{1}(\eta)$ belong to $(-r,r)$, and consider a bounded linear operator $T_{0}:H^{-r}(V_{0})\to H^{-r}(\mathbb{R}^{n})$ such that $T_{0}w=w$ in $V_{0}$ whenever $w\in H^{-r}(V_{0})$ and that its restriction to $H^{r}(V_{0})$ is a bounded operator $T_{0}:H^{r}(V_{0})\to H^{r}(\mathbb{R}^{n})$. Such an extension operator exists; see, e.g., \cite[Theorem 4.2.2]{Triebel95}. It follows from Theorem~\ref{prop5.1} that the restriction of $T_{0}$ to $H^{\varphi\varrho^{2q}}(V_0)$ or $H^{\eta}(V_0)$ is a bounded operator $T_{0}:H^{\varphi\varrho^{2q}}(V_0)\to H^{\varphi\varrho^{2q}}(\mathbb{R}^{n})$ or $T_{0}:H^{\eta}(V_0)\to H^{\eta}(\mathbb{R}^{n})$, resp. Then the linear mapping $T:w\mapsto(T_{0}w)\!\upharpoonright\!\Omega$ acts continuously between the following spaces: $T:H^{\varphi\varrho^{2q}}(V_0)\to H^{\varphi\varrho^{2q}}(\Omega)$ and $T:H^{\eta}(V_0)\to H^{\eta}(\Omega)$. Hence, $\zeta u_{k}=\zeta(Tv_{k})\to\zeta(Tv)=\zeta u$
in $H^{\varphi\varrho^{2q}}(\Omega)$ and $\zeta Au_{k}=\zeta(TAv_{k})\to\zeta(TAv)=\zeta Au$ in $H^{\eta}(\Omega)$ as $k\to\infty$; i.e., \eqref{f6.38} and \eqref{f6.39} hold true.

According to \eqref{f6.38}, we obtain
\begin{equation}\label{f6.40}
\chi u_{k}=\chi\zeta u_{k}\to\chi\zeta u=\chi u\quad\mbox{in}\;\;
H^{\varphi\varrho^{2q}}(\Omega)
\end{equation}
and
\begin{equation}\label{f6.41}
\zeta u_{k}\to\zeta u\quad\mbox{in}\;\;
H^{\varphi\varrho^{2q-\ell}}(\Omega).
\end{equation}
as $k\to\infty$. Let us show that
\begin{equation}\label{f6.42}
\zeta B_{j}u_{k}\to\zeta B_{j}u\quad\mbox{in}\;\;
H^{\varphi\varrho^{2q-m_j-1/2}}(\Gamma)
\end{equation}
as $k\to\infty$ for each $j\in\{1,\ldots,q\}$.

Given $j$, we consider a boundary PDO on $\partial V_0$ of the form
\begin{equation*}
B_{j}^{\star}:=B_{j}^{\star}(x,D):=
\sum_{|\mu|\leq m_j}b_{j,\mu}^{\star}(x)D^{\mu}
\end{equation*}
where every coefficient $b_{j,\mu}^{\star}$ belongs to $C^{\infty}(\partial V_0)$ and coincides with the corresponding coefficient $b_{j,\mu}$ of $B_{j}$ on $\Gamma\cap\partial V_0$. Since $v_{k}\to v$ in $H^{\varphi\varrho^{2q}}_{A,\eta}(V_0)$, we conclude by Theorem~\ref{th1} in view of Remark~\ref{bounded-operator} that
\begin{equation*}
B_{j}^{\star}v_{k}\to B_{j}^{\star}v\quad\mbox{in}\quad H^{\varphi\varrho^{{2q}-m_{j}-1/2}}(\partial V_0)
\end{equation*}
as $k\to\infty$. But $\zeta B_{j}^{\star}v_{k}=\zeta B_{j}u_{k}$ on $\Gamma\cap\partial V_0$ whenever $k\geq1$. Hence,
\begin{equation}\label{f6.43}
\zeta B_{j}u_{k}\to T_{1}(\zeta B_{j}^{\star}v)\quad\mbox{in}\;\;
H^{\varphi\varrho^{2q-m_j-1/2}}(\Gamma),
\end{equation}
where the distribution $T_{1}(\zeta B_{j}^{\star}v)$ is equal by definition to $\zeta B_{j}^{\star}v$ on $\Gamma\cap V$ and to zero on $\Gamma\setminus\mathrm{supp}\,\zeta$.

Remark that
\begin{equation}\label{f6.44}
\zeta B_{j}^{\star}v=\zeta B_{j}u\quad\mbox{on}\;\;\Gamma\cap V.
\end{equation}
Indeed, since $u\in\mathcal{S}'(\Omega)$ and $Au\in H^{-1/2+}(\Omega)$
by the hypotheses of Theorem~$\ref{th4.7}$, there exist numbers $\theta<-1/2$ and $\delta>-1/2$ such that $u\in H^{\theta}_{A,\delta}(\Omega)$. According to Theorem~\ref{th1}, there exists a sequence $(w_{k})_{k=1}^{\infty}\subset C^{\infty}(\overline{\Omega})$ that converges to $u$ in $H^{\theta}_{A,\delta}(\Omega)$. Then $w_{k}^{\circ}:=w_{k}\!\upharpoonright\!V_0\to u\!\upharpoonright\!V_0=v$ in $H^{\theta}_{A,\delta}(V_0)$. Hence, $\zeta B_{j}w_{k}\to\zeta B_{j}u$ in $H^{\theta-m_j-1/2}(\Gamma)$ and also $\zeta B_{j}^{\star}w_{k}^{\circ}\to\zeta B_{j}^{\star}v$ in $H^{\theta-m_j-1/2}(\partial V_0)$ as $k\to\infty$. However, $\zeta B_{j}w_{k}=\zeta B_{j}^{\star}w_{k}^{\circ}$ on $\Gamma\cap\partial V_0\supset\Gamma\cap V$. Hence, the last two limits imply property \eqref{f6.44}. Owing to this property, we have the equality of distributions $T_{1}(\zeta B_{j}^{\star}v)=\zeta B_{j}u$ on $\Gamma$, which together with \eqref{f6.43} gives \eqref{f6.42}.

Now we may complete the proof. According to Step~3, the inequality
\begin{equation*}
\|\chi u_k\|_{\varphi\varrho^{2q},\Omega}
\leq c\,\bigl(\|\zeta(A,B)u_k\|_{\eta,\varphi,\Omega,\Gamma}+
\|\zeta u_k\|_{\varphi\rho^{2q-\ell},\Omega}\bigr)
\end{equation*}
holds true for every $k\geq1$. Passing here to the limit as $k\to\infty$ and using \eqref{f6.39}--\eqref{f6.42}, we conclude that the estimate \eqref{f4.21} holds true under the hypotheses of Theorem~$\ref{th4.7}$.
\end{proof}

\begin{remark}\label{rem6.3}
We stated in Remark \ref{rem4.14} that Theorem \ref{th4.13} remains valid for every  $\varphi\in\mathrm{OR}$ subject to $\sigma_0(\varphi)>-1/2$ if we put $\eta:=\varphi$. The proof of this result is performed in the same way as the proof just given and is somewhat simpler. Namely, we may assume  that $-1/2<s_0<\sigma_0(\varphi)$; then $X_{2q}=H^{\varphi\varrho^{2q}}(\Omega)$ and $Y_0=H^{\varphi}(\Omega)$ up to equivalence of norms due to \eqref{f6.12}. This immediately implies
\eqref{f6.31} and \eqref{f6.33} on Step~2.
\end{remark}

\section{Applications to homogeneous elliptic equations}\label{sec7}

\subsection{Solvability and regularity theorems}\label{sec7.1}
Let us discuss applications of the theorems from Section~\ref{sec4} to the regular elliptic boundary value problem \eqref{f1}, \eqref{f2} in the important case where the elliptic equation \eqref{f1} is homogeneous, i.e. $f=0$ in $\Omega$. In this case, we may formulate versions of these theorems for every $\varphi\in\mathrm{OR}$. We will consider these versions for more general Theorems \ref{th1}, \ref{th2}, \ref{th4.7}, and \ref{th4.13} and then discuss the corresponding proofs. It is convenient to use the function parameter $\alpha:=\varphi\rho^{2q}$ in the case indicated.

Given $\alpha\in\mathrm{OR}$, we put
\begin{equation*}
H^{\alpha}_{A}(\Omega):=
\bigl\{u\in H^{\alpha}(\Omega):Au=0\;\,\mbox{in}\;\,\Omega\bigr\};
\end{equation*}
as usual, $Au$ is understood in the theory of distributions. We endow the linear space $H^{\alpha}_{A}(\Omega)$ with the inner product and norm in $H^{\alpha}(\Omega)$. The space $H^{\alpha}_{A}(\Omega)$ is complete with this norm because the differential operator $A$ is continuous on $\mathcal{D}'(\Omega)$. If $\alpha(t)\equiv t^{s}$ for a certain $s\in\mathbb{R}$, the space $H^{\alpha}_{A}(\Omega)$ is also denoted by $H^{s}_{A}(\Omega)$ according to our convention.

Since $A$ is elliptic on $\overline{\Omega}$, the inclusion $H^{\alpha}_{A}(\Omega)\subset C^{\infty}(\Omega)$ holds true (see, e.g., \cite[Chapter~2, Theorem~3.2]{LionsMagenes72}. However,  $H^{\alpha}_{A}(\Omega)\not\subset C^{\infty}(\overline{\Omega})$. Put
\begin{equation*}
C^{\infty}_{A}(\overline{\Omega}):=
\bigl\{u\in C^{\infty}(\overline{\Omega}):Au=0\;\,\mbox{on}\;\,
\overline{\Omega}\bigr\}.
\end{equation*}
With the problem \eqref{f1}, \eqref{f2} in the $f=0$ case, we associate the mapping
\begin{equation}\label{f7.1}
B_{A}:u\mapsto Bu=(B_{1}u,\ldots,B_{q}u),\quad\mbox{where}\quad
u\in C^{\infty}_{A}(\overline{\Omega}).
\end{equation}
Put
\begin{equation*}
N^{+}_1:=\bigl\{(C^{+}_{1}v,\ldots,C^{+}_{q}v):v\in N^{+}\bigr\}.
\end{equation*}
Of course, $\dim N^{+}_{1}\leq\dim N^{+}<\infty$. The inequality $\dim N^{+}_{1}<\dim N^{+}$ is possible, which follows from a result by Pli\'{s} \cite{Plis61} (this result is expounded in the book \cite[Theorem~13.6.15]{Hermander83}).

\begin{theorem}\label{th7.1}
Let $\alpha\in\mathrm{OR}$. Then the set $C^{\infty}_{A}(\overline{\Omega})$ is dense in the space  $H^{\alpha}_{A}(\Omega)$, and the mapping \eqref{f7.1} extends uniquely (by continuity) to a bounded linear operator
\begin{equation}\label{f7.2}
B_{A}:H^{\alpha}_{A}(\Omega)\to
\bigoplus_{j=1}^{q}H^{\alpha\rho^{-m_j-1/2}}(\Gamma)=:
\mathcal{H}_{\alpha}(\Gamma).
\end{equation}
This operator is Fredholm. Its kernel coincides with $N$, and its range  consists of all vectors $g\in\mathcal{H}_{\alpha}(\Gamma)$ such that
\begin{equation}\label{f7.3}
\sum_{j=1}^{q}\,(g_{j},\,C^{+}_{j}v)_{\Gamma}=0
\quad\mbox{for every}\quad v\in N^{+}.
\end{equation}
The index of the operator \eqref{f7.2} equals $\dim N-\dim N^{+}_1$ and does not depend on $\alpha$.
\end{theorem}

If $N=\{0\}$ and $N^{+}_1=\{0\}$, the operator \eqref{f7.2} is an isomorphism between the spaces $H^{\alpha}_{A}(\Omega)$ and $\mathcal{H}_{\alpha}(\Gamma)$. Generally, this operator induces an isomorphism which may be built with the help of the following decompositions of these spaces:
\begin{gather}\label{f7.4}
H^{\alpha}_{A}(\Omega)=N\dotplus\{u\in
H^{\alpha}_{A}(\Omega):(u,w)_\Omega=0\;\,\mbox{for every}\;\,w\in N\},\\
\mathcal{H}_{\alpha}(\Gamma)=N^{+}_{1}\dotplus
\{(g_1,\ldots,g_q)\in\mathcal{H}_{\alpha}(\Gamma):\mbox{\eqref{f7.3} is true}\}.\label{f7.5}
\end{gather}
These formulas need commenting. If $\sigma_{0}(\alpha)>0$, the second summand in \eqref{f7.4} is well defined and is closed in $H^{\alpha}_{A}(\Omega)$ due to the continuous embedding $H^{\alpha}_{A}(\Omega)\hookrightarrow L_{2}(\Omega)$. Hence, in this case, \eqref{f7.4} is the restriction to $H^{\alpha}_{A}(\Omega)$ of the corresponding decomposition of $L_{2}(\Omega)$ into the orthogonal sum of subspaces. If $\sigma_{0}(\alpha)\leq0$, we use Lemma~\ref{lema1} and also formula \eqref{sum1} for $\alpha=\varphi\varrho^{2q}$ and $\omega=\eta$. In this case, the second summand in \eqref{f7.4} is well defined and closed in $H^{\alpha}_{A}(\Omega)$ according to this lemma, and \eqref{f7.4} is the restriction of the decomposition \eqref{sum1} to $H^{\alpha}_{A}(\Omega)$. Formula \eqref{f7.5} is true because the summands on the right have the trivial intersection, and the finite dimension of the first summand coincides with the codimension of the second. Indeed, since $\mathcal{H}_{\alpha}(\Gamma)$ is dual to $\mathcal{H}_{1/\alpha}(\Gamma)$ with respect to the form $(\cdot,\cdot)_{\Gamma}+\cdots+(\cdot,\cdot)_{\Gamma}$ (this is proved analogously to \cite[Theorem~2.3(v)]{MikhailetsMurach14}), the dimension of the dual of $N_{1}^{+}\subset\mathcal{H}_{1/\alpha}(\Gamma)$ equals the above-mentioned codimension.

Let $P_{1}$ and $\mathcal{P}^+_{1}$ respectively denote the projectors of the spaces $H^{\alpha}_{A}(\Omega)$ and $\mathcal{H}_{\alpha}(\Gamma)$ onto the second summand in \eqref{f7.4} and \eqref{f7.5} parallel to the first. The mappings  defining these projectors do not depend on $\alpha$. Note that $P_{1}$ is a restriction of the projector $P$ from Theorem~\ref{th2} where $\alpha=\varphi\varrho^{2q}$ and $\sigma_{0}(\alpha)\leq2q-1/2$.

\begin{theorem}\label{th7.2}
Let $\alpha\in\mathrm{OR}$. Then the restriction of the operator \eqref{f7.2} to the second summand in \eqref{f7.4} is an isomorphism
\begin{equation}\label{f7.6}
B_{A}:P_{1}(H^{\alpha}_{A}(\Omega))\leftrightarrow
\mathcal{P}^{+}_{1}(\mathcal{H}_{\alpha}(\Gamma)).
\end{equation}
\end{theorem}

Let us turn to properties of generalized solutions to the elliptic problem \eqref{f1}, \eqref{f2} in the case where $f=0$ in $\Omega$. The notion of a generalized solution introduced just before Theorem~\ref{th4.6} is applicable in this case. Since every solution $u$ to the homogeneous elliptic equation \eqref{f1} belongs to $C^{\infty}(\Omega)$, we are interested in properties of $u=u(x)$ when the argument $x$ approaches the boundary $\Gamma$ of $\Omega$. Let $\Gamma_{0}$ be a nonempty open subset of $\Gamma$.

\begin{theorem}\label{th7.3}
Let $\alpha\in\mathrm{OR}$. Assume that a distribution $u\in\mathcal{S}'(\Omega)$ is a generalized solution to the elliptic problem \eqref{f1}, \eqref{f2} whose right-hand sides satisfy the conditions $f=0$ in $\Omega$ and
\begin{equation}\label{f7.7}
g_j\in H^{\alpha\rho^{-m_j-1/2}}_{\mathrm{loc}}(\Gamma_0)
\quad\mbox{for each}\quad j\in\{1,\ldots,q\}.
\end{equation}
Then $u\in H^{\alpha}_{\mathrm{loc}}(\Omega,\Gamma_{0})$.
\end{theorem}

We supplement this theorem with a corresponding estimate of $u$. Let $\|\cdot\|_{\alpha,\Gamma}'$ denote the norm in the Hilbert space $\mathcal{H}_{\alpha}(\Gamma)$ defined in \eqref{f7.2}.

\begin{theorem}\label{th7.4}
Let $\alpha\in\mathrm{OR}$, and assume that a distribution $u\in\mathcal{S}'(\Omega)$ satisfies the hypotheses of Theorem~$\ref{th7.3}$. We arbitrarily choose a number $\ell>0$ and functions $\chi,\zeta\in C^{\infty}(\overline{\Omega})$ such that $\mathrm{supp}\,\chi\subset\mathrm{supp}\,\zeta\subset\Omega\cup\Gamma_{0}$ and that $\zeta=1$ in a neighbourhood of $\mathrm{supp}\,\chi$. Then
\begin{equation}\label{f7.8}
\|\chi u\|_{\alpha,\Omega}\leq c\,\bigl(\|\zeta g\|_{\alpha,\Gamma}'+
\|\zeta u\|_{\alpha\rho^{-\ell},\Omega}\bigr)
\end{equation}
for some number $c>0$ that does not depend on $u$ and $g$.
\end{theorem}

Let us discuss the proofs of these theorems. Theorem~\ref{th7.1} follows from Proposition~\ref{prop1} and Theorem~\ref{th1} excepting the conclusion about the density of $C^{\infty}_{A}(\overline{\Omega})$ in   $H^{\alpha}_{A}(\Omega)$. Indeed, the restriction of the Fredholm operator \eqref{prop1} if $\sigma_0(\varphi)>-1/2$ or the Fredholm operator \eqref{f15} if $\sigma_0(\varphi)\leq-1/2$ to the space $H^{\alpha}_{A}(\Omega)$, where $\alpha=\varphi\varrho^{2q}$, is evidently a Fredholm bounded operator between the spaces \eqref{f7.2} with indicated properties of its kernel, range, and index. Theorem~\ref{th7.2} is a direct consequence of this part of Theorem~\ref{th7.1} and the Banach theorem on inverse operator. The mentioned density is easily deduced from Theorem~\ref{th7.2}. Indeed, since the set $(C^{\infty}(\Gamma))^{q}$ is dense in $\mathcal{H}_{\alpha}(\Gamma)$ for every $\alpha\in\mathrm{OR}$, its subset $\mathcal{P}^{+}_{1}((C^{\infty}(\Gamma))^{q})$ is dense in the range of the isomorphism \eqref{f7.6}. Hence, the set $B_{A}^{-1}\mathcal{P}^{+}_{1}((C^{\infty}(\Gamma))^{q})$ lies in $C^{\infty}_{A}(\overline{\Omega})$ and is dense in the subspace $P_{1}(H^{\alpha}_{A}(\Omega))$ of $H^{\alpha}_{A}(\Omega)$; here, $B_{A}^{-1}$ denotes the inverse of \eqref{f7.6}. This yields the required density of $C^{\infty}_{A}(\overline{\Omega})$ in $H^{\alpha}_{A}(\Omega)=N\dotplus P_{1}(H^{\alpha}_{A}(\Omega))$.
Theorem~\ref{th7.3} follows immediately from
Theorem~\ref{th4.7} and Remark~\ref{rem4.11}. Theorem~\ref{th7.4} is a direct consequence of Theorem~\ref{th4.13} and Remark~\ref{rem4.14}.

Note that Theorems \ref{th7.1} and \ref{th7.2} are established in our paper \cite[Section~4]{AnopMurach16Coll2}, whereas Theorem~\ref{th7.3} is announced in \cite[Section~3]{AnopMurach18Dop3} (without proof), these papers being published in Ukrainian. If the function $\alpha$ is regularly varying at infinity, Theorem \ref{th7.1} is proved in \cite[Section~1]{MikhailetsMurach06UMJ11} (see also the monograph \cite[Section~3.3.1]{MikhailetsMurach14}). This theorem is a classical result in the Sobolev case where $\alpha(t)\equiv t^{s}$; see, e.g., the book \cite[Chapter~2, Section~7.3]{LionsMagenes72}, which contains this theorem if $s\in\mathbb{R}\setminus\{-1/2,-3/2,\ldots\}$. In this connection, we mention Seeley's paper \cite{Seeley66}, which investigates the Cauchy data of functions from $H^{s}_{A}(\Omega)$ where $s\in\mathbb{R}$ (see also the survey \cite[Section 5.4~b]{Agranovich97}).

\subsection{Uniform convergence of solutions}\label{sec7.2}
Using generalized Sobolev spaces over $\Gamma$, we obtain a sufficient condition for the uniform convergence of solutions to the elliptic equation $Au=0$ and their derivatives of a prescribed order.

\begin{theorem}\label{th7.5}
Let $0\leq p\in\mathbb{Z}$. Assume that a sequence $(u_{k})_{k=1}^{\infty}\subset\mathcal{S}'(\Omega)$ satisfies the following two conditions: $Au_{k}=0$ in $\Omega$ whenever $k\geq1$, and the sequence of the distributions $g^{(k)}:=B_{A}u_{k}$ converges in the space $\mathcal{H}_{\alpha}(\Gamma)$ for some $\alpha\in\mathrm{OR}$ subject to
\begin{equation}\label{f7.9}
\int\limits_1^{\infty} t^{2p+n-1}\alpha^{-2}(t)\,dt<\infty.
\end{equation}
Then every $u_{k}\in C^{p}(\overline{\Omega})$, and there exists a function $u\in C^{p}(\overline{\Omega})$ that
the sequence $(D^{\mu}P_{1}u_{k})_{k=1}^{\infty}$ converges uniformly to $D^{\mu}u$ on $\overline{\Omega}$ whenever $|\mu|\leq p$. The function $u$ satisfies the conditions $Au=0$ in $\Omega$ and $B_{A}u=g$ on $\Gamma$, where $g$ is the limit of the sequence $(g^{(k)})_{k=1}^{\infty}$.
\end{theorem}

In this theorem, the vectors $B_{A}u_{k}$ and $B_{A}u$ are well defined by means of the operator \eqref{f7.2} because $u_{k}\in H^{-r}_{A}(\Omega)$ whenever $r\gg1$ due to the hypotheses of the theorem and because $u\in H^{p}_{A}(\Omega)$ due to its conclusion. If $p\leq m_{j}-1$, the smoothness of $u_{k}$ and $u$ is not sufficient to find the $j$-th components of $B_{A}u_{k}$ and $B_{A}u$ with the help of classical derivatives. Recall that the hypothesis $Au_{k}=0$ in $\Omega$ and the conclusion $Au=0$ in $\Omega$ are understood in the distribution theory sense and imply the inclusions of $u_{k}$ and $u$ in $C^{\infty}(\Omega)$. Note if $N=\{0\}$, then $P_{1}u_{k}=u_{k}$.

It is useful to compare this theorem with the classical Harnack theorem on the uniform convergence of a sequence of harmonic functions on $\overline{\Omega}$ (see, e.g., \cite[Section~2.6]{GilbargTrudinger98}).  The latter theorem (also called the Bauer convergence property) relates to the case where $A$ is the Laplace operator, $B_{A}u:=u\!\upharpoonright\!\Gamma$ for every $u\in C^{\infty}_{A}(\overline{\Omega})$, and $p=0$ in Theorem~\ref{th7.5}. In this case, $H_{\alpha}(\Gamma)=H^{\alpha\varrho^{-1/2}}(\Gamma)\hookrightarrow C(\Gamma)$ due to condition \eqref{f7.9} and property \eqref{Hermander-embedding} considered for $C(\mathbb{R}^{n-1})$ instead of $C^{p}(\mathbb{R}^{n})$. Hence, the conclusion of Theorem~\ref{th7.5} about the uniform convergence of the sequence of harmonic functions $u_{k}$ follows from the Harnak theorem. However, for first-order boundary conditions, Theorem~\ref{th7.5} gives weak enough and new sufficient conditions for this convergence. Thus, considering the Neumann boundary condition, we conclude by Theorem~\ref{th7.5} that the sequence of harmonic functions $u_{k}\in\mathcal{S}'(\Omega)$ subject to $(u_{k},1)_{\Omega}=0$ converges uniformly on $\overline{\Omega}$ if the sequence of traces of their normal derivatives converges in the space $H^{\omega}(\Gamma)$ where $\omega(t):=t^{(n-3)/2}\log(1+t)$ whenever $t\geq1$, e.g. In the $n=2$ case, this space is broader than $H^{-1/2+}(\Gamma)$.

Consider a version of Theorem~\ref{th7.5} for an open subset $\Gamma_{0}\neq\emptyset$ of the boundary $\Gamma$. Given $\alpha\in\mathrm{OR}$, we introduce the linear space
\begin{equation*}
\mathcal{H}_{\alpha}(\Gamma_{0}):=\{g\!\upharpoonright\!\Gamma_{0}:g\in \mathcal{H}_{\alpha}(\Gamma)\}
\end{equation*}
endowed with the norm
\begin{equation*}
\|h\|_{\alpha,\Gamma_{0}}':=
\inf\bigl\{\,\|g\|_{\alpha,\Gamma}':
g\in\mathcal{H}_{\alpha}(\Gamma),\;g=h\;\,\mbox{in}\;\,\Gamma_{0}\bigr\}
\end{equation*}
of $h\in\mathcal{H}_{\alpha}(\Gamma_{0})$.

\begin{theorem}\label{th7.6}
Let $0\leq p\in\mathbb{Z}$. Assume that a sequence $(u_{k})_{k=1}^{\infty}\subset\mathcal{S}'(\Omega)$ satisfies the following three conditions: $Au_{k}=0$ in $\Omega$ whenever $k\geq1$, this sequence converges in the Sobolev space $H^{-r}(\Omega)$ if $r\gg1$, and the sequence of the distributions $(B_{A}u_{k})\!\upharpoonright\!\Gamma_{0}$ converges in the space $\mathcal{H}_{\alpha}(\Gamma_{0})$ for some $\alpha\in\mathrm{OR}$ subject to \eqref{f7.9}. Then every $u_{k}\in C^{p}(\Omega\cup\Gamma_{0})$, and each sequence  $(D^{\mu}u_{k})_{k=1}^{\infty}$, with $|\mu|\leq p$, converges uniformly on every closed (in $\mathbb{R}^{n}$) subset of $\Omega\cup\Gamma_{0}$.
\end{theorem}

\begin{proof}[Proof of Theorem $\ref{th7.5}$.]
By Theorem~\ref{th7.3} in the $\Gamma_{0}=\Gamma$ case and by the H\"ormander embedding theorem \eqref{Hermander-embedding}, we conclude that each $u_{k}\in H^{\alpha}(\Omega)\hookrightarrow C^{p}(\overline{\Omega})$. Since $g^{(k)}\to g$ in the subspace $\mathcal{P}^{+}_{1}(\mathcal{H}_{\alpha}(\Gamma))$ of $\mathcal{H}_{\alpha}(\Gamma)$, Theorem~\ref{th7.2} implies the convergence $P_{1}u_{k}\to u$ in $H^{\alpha}(\Omega)$, where $u\in P_{1}(H^{\alpha}_{A}(\Omega))$ is the inverse image of $g$ under the isomorphism \eqref{f7.6}. This gives the conclusion of Theorem~\ref{th7.5}  due to the continuous embedding $H^{\alpha}(\Omega)$ in $C^{p}(\overline{\Omega})$.
\end{proof}

\begin{proof}[Proof of Theorem $\ref{th7.6}$.]
Since every $(B_{A}u_{k})\!\upharpoonright\!\Gamma_{0}$ belongs to $\mathcal{H}_{\alpha}(\Gamma_{0})$, we have the inclusion
\begin{equation*}
B_{A}u_{k}\in\prod_{j=1}^{q}
H^{\alpha\rho^{-m_j-1/2}}_{\mathrm{loc}}(\Gamma_0).
\end{equation*}
Hence, every $u_{k}\in H^{\alpha}_{\mathrm{loc}}(\Omega,\Gamma_{0})$ due to Theorem~\ref{th7.3}, which implies by \eqref{Hermander-embedding} that $u_{k}\in C^{p}(\Omega\cup\Gamma_{0})$. By the hypotheses of Theorem~\ref{th7.6}, we have
\begin{equation}\label{f7.10}
u_{k}\to u\quad\mbox{in}\quad H^{-r}(\Omega)\quad\mbox{for some}\quad
u\in H^{-r}_{A}(\Omega)
\end{equation}
and
\begin{equation}\label{f7.11}
(B_{A}u_{k})\!\upharpoonright\!\Gamma_{0}\to (B_{A}u)\!\upharpoonright\!\Gamma_{0}\quad\mbox{in}\quad \mathcal{H}_{\alpha}(\Gamma_{0})
\end{equation}
as $k\to\infty$. Thus, the distribution $u_{k}-u$ satisfies the hypotheses of Theorem~\ref{th7.3} in which we take $u_{k}-u$ instead of $u$. Hence, we may apply Theorem~\ref{th7.4} to $u_{k}-u$. Let $G$ be a nonempty closed subset of $\Omega\cup\Gamma_{0}$. Choose functions $\chi,\zeta\in C^{\infty}(\overline{\Omega})$ such that $\mathrm{supp}\,\chi\subset \mathrm{supp}\,\zeta\subset\Omega\cup\Gamma_{0}$, $\chi=1$ in a neighbourhood of $G$, and $\zeta=1$ in a neighbourhood of $\mathrm{supp}\,\chi$. According to Theorem~\ref{th7.4}, we have the inequality
\begin{equation*}
\|\chi(u_{k}-u)\|_{\alpha,\Omega}\leq c\,
\bigl(\|\zeta B_{A}(u_{k}-u)\|_{\alpha,\Gamma}'+
\|\zeta(u_{k}-u)\|_{-r,\Omega}\bigr),
\end{equation*}
where the number $c>0$ does not depend on $u_{k}-u$. Hence, $\chi u_{k}\to\chi u$ in $H^{\alpha}(\Omega)$ as $k\to\infty$ due to \eqref{f7.10} and \eqref{f7.11}. Therefore, $\chi u_{k}\to\chi u$ in $C^{p}(\overline{\Omega})$ by \eqref{Hermander-embedding}, which implies that the sequence $(D^{\mu}u_{k})_{k=1}^{\infty}$ converges uniformly on $G$ whenever $|\mu|\leq p$.
\end{proof}

\begin{remark}\label{rem7.7}
Considering Theorems \ref{th7.5} and \ref{th7.6}, it is useful to take into account the following property: if a sequence $(u_{k})_{k=1}^{\infty}\subset\mathcal{S}'(\Omega)$ satisfies the first two conditions formulated in Theorem~\ref{th7.6}, the sequence $(D^{\mu}u_{k})_{k=1}^{\infty}$ will converge uniformly on every closed subset $G$ of $\Omega$ for every multi-index $\mu$. This fact is known and follows from the internal a priory estimate
\begin{equation}\label{f7.12}
\|\chi(u_{k}-u)\|_{\ell,\Omega}\leq c_0\|u_{k}-u\|_{-r,\Omega}\to0
\quad\mbox{as}\quad k\to\infty
\end{equation}
in Sobolev spaces. Here, $\ell\gg1$, $u$ is the limit of the sequence $(u_{k})_{k=1}^{\infty}$, $\chi\in C^{\infty}(\overline{\Omega})$, $\mathrm{supp}\,\chi\subset\Omega$, $\chi=1$ in a neighbourhood of $G$,
and $c_0$ is some positive number that does not depend on $u_{k}-u$. It follows from \eqref{f7.12} by the Sobolev embedding theorem, that $\chi u_{k}\to\chi u$ in $C^{p}(\overline{\Omega})$ whenever $0\leq p\in\mathbb{Z}$, which yields the uniform convergence of $(D^{\mu}u_{k})_{k=1}^{\infty}$ on $G$ for every $\mu$. The estimate \eqref{f7.12} is known (see, e.g., \cite[Theorem~7.2.2]{Roitberg96}).
\end{remark}

\subsection{Interpolation properties of related spaces}\label{sec7.3}
Consider the Hilbert spaces $H^{\alpha}_{A}(\Omega)$, where $\alpha\in\mathrm{OR}$, formed by solutions to the homogeneous elliptic equation $Au=0$ in $\Omega$. These spaces have analogous interpolation properties to that of $H^{\alpha}(\Omega)$.

\begin{theorem}\label{th7.8}
\begin{itemize}
  \item[(i)] Under the hypotheses of Proposition~$\ref{prop5.1}$, we have
\begin{equation*}
[H^{r_0}_{A}(\Omega),H^{r_1}_{A}(\Omega)]_{\psi}=H^{\alpha}_{A}(\Omega)
\end{equation*}
up to equivalence of norms.
  \item[(ii)] Under the hypotheses of Proposition~$\ref{prop5.2}$, we have
\begin{equation*}
[H^{\alpha_0}_{A}(\Omega),H^{\alpha_1}_{A}(\Omega)]_{\psi}=
H^{\alpha}_{A}(\Omega)
\end{equation*}
up to equivalence of norms.
\end{itemize}
\end{theorem}

This theorem shows that the class of spaces
\begin{equation}\label{f7.13}
\{H^{\alpha}(\Omega):\alpha\in\mathrm{OR}\}
\end{equation}
is obtained by the interpolation with a function parameter between their Sobolev analogs and is closed with respect to the interpolation with a function parameter between Hilbert spaces.

\begin{theorem}\label{th7.9}
Let $r_{0},r_{1}\in\mathbb{R}$ and $r_{0}<r_{1}$. A Hilbert space $H$ is an interpolation space between the spaces $H^{r_{0}}_{A}(\Omega)$ and $H^{r_{1}}_{A}(\Omega)$ if and only if $H=H^{\alpha}_{A}(\Omega)$ up to equivalence of norms for some function parameter $\alpha\in\mathrm{OR}$ that satisfies condition~\eqref{f3.2}.
\end{theorem}

Of course, we mean in this theorem that the numbers $c_{0}$ and $c_{1}$ in condition \eqref{f3.2} do not depend on $t$ and $\lambda$. This condition is equivalent to the following pair of conditions:
\begin{enumerate}
\item [$\mathrm{(i)}$] $r_{0}\leq\sigma_{0}(\varphi)$ and, moreover,
$r_{0}<\sigma_{0}(\varphi)$ if the supremum in $\eqref{f3.2sup}$ is not attained;
\item [$\mathrm{(ii)}$] $\sigma_{1}(\varphi)\leq r_{1}$ and, moreover,
$\sigma_{1}(\varphi)<r_{1}$ if the infimum in $\eqref{f3.2inf}$ is not attained.
\end{enumerate}

Theorem~$\ref{th7.9}$ reveals that the class \eqref{f7.13} coincides up to equivalence of norms  with the class of all Hilbert spaces that are interpolation ones between the Sobolev spaces $H^{r_{0}}_{A}(\Omega)$ and $H^{r_{1}}_{A}(\Omega)$ where $r_0,r_1\in\mathbb{R}$ and $r_0<r_1$. If we omit the subscript $A$ in the formulation of Theorem~$\ref{th7.9}$, we will obtain the corresponding interpolation property of the class  $\{H^{\alpha}(\Omega):\alpha\in\mathrm{OR}\}$ proved in \cite[Theorem~2.4]{MikhailetsMurach15ResMath1}.

\begin{proof}[Proof of Theorem $\ref{th7.8}$.]
Assertion (i) is a direct consequence of (ii). Let us prove (ii). Consider the isomorphisms \eqref{f7.6} where $\alpha\in\{\alpha_0,\alpha_1\}$ and interpolate them with the function parameter~$\psi$. Since $\psi$ is an interpolation parameter, we conclude that the restriction of the mapping  \eqref{f7.6}, where $\alpha=\alpha_0$, is an isomorphism
\begin{equation}\label{f7.14}
B_{A}:\bigl[P_{1}(H^{\alpha_0}_{A}(\Omega)),
P_{1}(H^{\alpha_1}_{A}(\Omega))\bigr]_{\psi}\leftrightarrow
\bigl[\mathcal{P}^{+}_{1}(\mathcal{H}_{\alpha_0}(\Gamma)),
\mathcal{P}^{+}_{1}(\mathcal{H}_{\alpha_0}(\Gamma))\bigr]_{\psi}.
\end{equation}
According to Proposition~\ref{prop5.2} and the theorem on interpolation of subspaces \cite[Theorem~1.6]{MikhailetsMurach14}, the range of \eqref{f7.14} equals
\begin{equation*}
[\mathcal{H}_{\alpha_0}(\Gamma),\mathcal{H}_{\alpha_1}(\Gamma)]_{\psi}\cap
\mathcal{P}^{+}_{1}(\mathcal{H}_{\alpha_0}(\Gamma))=
\mathcal{H}_{\alpha}(\Gamma)\cap \mathcal{P}^{+}_{1}(\mathcal{H}_{\alpha_0}(\Gamma)) =\mathcal{P}^{+}_{1}(\mathcal{H}_{\alpha}(\Gamma)).
\end{equation*}
Hence,
\begin{equation}\label{f7.15}
\bigl[P_{1}(H^{\alpha_0}_{A}(\Omega)),
P_{1}(H^{\alpha_1}_{A}(\Omega))\bigr]_{\psi}=
P_{1}(H^{\alpha}_{A}(\Omega))
\end{equation}
due to the isomorphisms \eqref{f7.6} and \eqref{f7.14}. All these equalities of Hilbert spaces hold true up to equivalence of norms.

Given $\omega\in\mathrm{OR}$, we let $\widetilde{H}^{\omega}_{A}(\Omega)$ denote the linear space $H^{\omega}_{A}(\Omega)$ endowed with the equivalent inner product
\begin{equation*}
(P_{1}u,P_{1}v)_{\omega,\Omega}+(u-P_{1}u,u-P_{1}v)_{\omega,\Omega}
\end{equation*}
of functions $u,v\in H^{\omega}_{A}(\Omega)$. Now $\widetilde{H}^{\omega}_{A}(\Omega)$ equals the orthogonal sum $N\oplus P_{1}(H^{\omega}_{A}(\Omega))$. Hence,
\begin{align*}
[H^{\alpha_0}_{A}(\Omega),H^{\alpha_1}_{A}(\Omega)]_{\psi}&=
[\widetilde{H}^{\alpha_0}_{A}(\Omega),
\widetilde{H}^{\alpha_1}_{A}(\Omega)]_{\psi}=
[N,N]_{\psi}\oplus\bigl[P_{1}(H^{\alpha_0}_{A}(\Omega)),
P_{1}(H^{\alpha_1}_{A}(\Omega))\bigr]_{\psi}\\
&=N\oplus P_{1}(H^{\alpha}_{A}(\Omega))=\widetilde{H}^{\alpha}_{A}(\Omega)
=H^{\alpha}_{A}(\Omega)
\end{align*}
up to equivalence of norms due to \eqref{f7.15} and the theorem on interpolation of orthogonal sums of spaces \cite[Theorem~1.5]{MikhailetsMurach14}. Assertion (ii) is proved.
\end{proof}

\begin{proof}[Proof of Theorem $\ref{th7.9}$. Necessity.]
Let a Hilbert space $H$ be an interpolation space between $H^{r_{0}}_{A}(\Omega)$ and $H^{r_{1}}_{A}(\Omega)$. We then conclude by Ovchinnikov's theorem \cite[Theorem 11.4.1]{Ovchinnikov84} that $H=[H^{r_{0}}_{A}(\Omega),H^{r_{1}}_{A}(\Omega)]_{\psi}$ up to equivalence of norms for some interpolation parameter $\psi\in\mathcal{B}$. Hence, $H=H^{\alpha}_{A}(\Omega)$ according to Theorem \ref{th7.8}(ii), where the function $\alpha(t):=t^{r_{0}}\,\psi(t^{r_{1}-r_{0}})$ of $t\geq1$ belongs to $\mathrm{OR}$. This function satisfies \eqref{f3.2} due to \cite[Theorem 4.2]{MikhailetsMurach15ResMath1}. The necessity is proved.

\textit{Sufficiency.} Assume that a Hilbert space $H$ coincides up to equivalence of norms with the space $H^{\alpha}_{A}(\Omega)$ for some $\alpha\in\mathrm{OR}$ subject to \eqref{f3.2}. Define the function $\psi\in\mathcal{B}$ by formula \eqref{f5.1}. Since $\alpha(t)=t^{r_{0}}\,\psi(t^{r_{1}-r_{0}})$ whenever $t\geq1$, the function $\psi$ is an interpolation parameter by \cite[Theorem 4.2]{MikhailetsMurach15ResMath1}. Therefore, $H$ equals $[H^{r_0}_{A}(\Omega),H^{r_1}_{A}(\Omega)]_{\psi}$
up to equivalence of norms due to Theorem \ref{th7.8}(ii). Thus, $H$ is an interpolation space between $H^{r_{0}}_{A}(\Omega)$ and $H^{r_{1}}_{A}(\Omega)$. The sufficiency is also proved.
\end{proof}

\section{Application to elliptic problems with boundary white noise}\label{sec8}

In this section, we apply the above results to some elliptic problems with rough boundary data induced by white noise. In particular, we are interested in boundary data belonging to the Nikolskii space $B_{p,\infty}^s(\Gamma)$ with $s<0$ and $p=2$ (see \cite[Sections 2.3.1 and 4.7.1]{Triebel95} and references therein on works by Nikolskii, e.g. \cite[Section 4.3.3]{Nikolskii77}, who introduced and investigated the space $B_{p,\infty}^s(\mathbb{R}^{n})$ for $s>0$ and $1\leq p\leq\infty$). This is motivated by recent results on Gaussian white noise; see below for details. We start with an embedding result.

\begin{proposition}\label{8.1}
Let $1\leq n\in\mathbb{Z}$, $s\in\mathbb{R}$, and $\alpha\in\mathrm{OR}$. Then the condition
\begin{equation}\label{int-cond-8.1}
\int\limits_{1}^{\infty}\frac{\alpha^{2}(t)}{t^{2s+1}}\,dt<\infty.
\end{equation}
is equivalent to the continuous embedding
\begin{equation}\label{embed}
B^{s}_{2,\infty}(\mathbb{R}^{n})\hookrightarrow H^{\alpha}(\mathbb{R}^{n}).
\end{equation}
\end{proposition}

This proposition is implicitly contained in Gol'dman's result \cite[Chapter~1, Theorem~2]{Goldman84}. We will give a proof of this proposition for the reader's convenience.

\begin{proof}[Proof of Proposition~$\ref{8.1}$] First we will treat the $s>0$ case and then reduce the $s\leq0$ case to the previous one. Put $Q_{0}:=\{\xi\in\mathbb{R}^{n}:|\xi|\leq1\}$ and
$Q_{k}:=\{\xi\in\mathbb{R}^{n}:2^{k-1}<|\xi|\leq2^{k}\}$ whenever $1\leq k\in\mathbb{Z}$. Let $s>0$; then the Nikolskii space $B^{s}_{2,\infty}(\mathbb{R}^{n})$ consists of all functions $w\in L_{2}(\mathbb{R}^{n})$ such that
\begin{equation*}
\|w\|_{s,\infty,\mathbb{R}^{n}}^{2}:=\sup_{0\leq k\in\mathbb{Z}}4^{sk}\int\limits_{Q_{k}}|\widehat{w}(\xi)|^{2}d\xi
<\infty,
\end{equation*}
with the norm in this space being equivalent to $\|\cdot\|_{s,\infty,\mathbb{R}^{n}}$; see, e.g., \cite[Lemma 2.11.2]{Triebel95}.

Assume that condition \eqref{int-cond-8.1} is satisfied. Given $w\in B^{s}_{2,\infty}(\mathbb{R}^{n})$, we have
\begin{equation}\label{norms-inequality}
\begin{aligned}
\|w\|_{\alpha,\mathbb{R}^{n}}^{2}&=\sum_{k=0}^{\infty}\,
\int\limits_{Q_k}\alpha^{2}(\langle\xi\rangle)\,|\widehat{w}(\xi)|^{2}d\xi
\asymp\sum_{k=0}^{\infty}\alpha^{2}(2^{k})
\int\limits_{Q_k}|\widehat{w}(\xi)|^{2}d\xi\\
&\leq\biggl(\sum_{k=0}^{\infty}
\frac{\alpha^{2}(2^{k})}{4^{sk}}\biggr)
\sup_{0\leq k\in\mathbb{Z}}4^{sk}
\int\limits_{Q_k}|\widehat{w}(\xi)|^{2}d\xi
=c\,\|w\|_{s,\infty,\mathbb{R}^{n}}^{2}<\infty.
\end{aligned}
\end{equation}
Here, the symbol "$\asymp$" means the equivalence of norms squared, this equivalence being true by \eqref{f3.1} in the $b=2$ case. Besides,
\begin{equation}\label{equivalence}
c:=\sum_{k=0}^{\infty}\frac{\alpha^{2}(2^{k})}{4^{sk}}<\infty\quad
\Longleftrightarrow\quad\eqref{int-cond-8.1}
\end{equation}
because the function $\alpha^{2}(t)\,t^{-2s}$ of $t\geq1$ belongs to $\mathrm{OR}$. Indeed, if $\omega\in\mathrm{OR}$, then
\begin{align*}
\int\limits_{1}^{\infty}\frac{\omega(t)}{t}dt&=
\sum_{k=0}^{\infty}\int\limits_{2^{k}}^{2^{k+1}}\frac{\omega(t)}{t}dt=
\sum_{k=0}^{\infty}\int\limits_{1}^{2}\frac{\omega(2^k\tau)}{\tau}d\tau\\
&=\sum_{k=0}^{\infty}\omega(2^k)
\int\limits_{1}^{2}\frac{\omega(2^k\tau)}{\omega(2^k)}\frac{d\tau}{\tau}
\asymp\sum_{k=0}^{\infty}\omega(2^k)\int\limits_{1}^{2}\frac{d\tau}{\tau}.
\end{align*}
This implies that
\begin{equation}\label{equivalence-gen}
\int\limits_{1}^{\infty}\frac{\omega(t)}{t}dt<\infty\quad
\Longleftrightarrow\quad\sum_{k=0}^{\infty}\omega(2^k)<\infty
\end{equation}
for every $\omega\in\mathrm{OR}$. Now \eqref{equivalence-gen} written for $\omega(t)\equiv\alpha^{2}(t)\,t^{-2s}$ is \eqref{equivalence}. Thus, it follows from \eqref{norms-inequality} that condition \eqref{int-cond-8.1} implies the continuous embedding \eqref{embed}.

Let us prove the inverse implication. We define a function $v\in L_{2}(\mathbb{R}^{n})$ as follows: $v(\xi):=2^{-sk}(\mathrm{mes}\,Q_{k})^{-1/2}$ if $\xi\in Q_{k}$ for some integer $k\geq0$. The function $w:=\mathcal{F}^{-1}v_{m}$ belongs to $B^{s}_{2,\infty}(\mathbb{R}^{n})$, and $\|w\|_{s,\infty,\mathbb{R}^{n}}=1$, where $\mathcal{F}^{-1}$ is the inverse Fourier transform. Assume that the continuous embedding \eqref{embed} holds true. In view of \eqref{norms-inequality}, we have
\begin{align*}
\sum_{k=0}^{\infty}\frac{\alpha^{2}(2^{k})}{4^{sk}}=
\sum_{k=0}^{\infty}\alpha^{2}(2^{k})
\int\limits_{Q_k}|\widehat{w}(\xi)|^{2}d\xi
\asymp\|w\|_{\alpha,\mathbb{R}^{n}}^{2}\leq c_{0}^{2}\,\|w\|_{s,\infty,\mathbb{R}^{n}}^{2}<\infty,
\end{align*}
where $c_{0}$ is the norm of the continuous embedding operator \eqref{embed}. Thus, this embedding implies condition \eqref{int-cond-8.1}.

We have proved the equivalence $\eqref{int-cond-8.1}\Leftrightarrow\eqref{embed}$ in the $s>0$ case. The $s\leq0$ case is plainly reduced to the case considered with the help of the fact that the mapping $w\mapsto\mathcal{F}^{-1}[\langle\xi\rangle^{-\lambda}\widehat{w}(\xi)]$ sets topological isomorphisms
$H^{\alpha}(\mathbb{R}^{n})\leftrightarrow H^{\alpha\varrho^{\lambda}}(\mathbb{R}^{n})$ and $B^{s}_{2,\infty}(\mathbb{R}^{n})\leftrightarrow B^{s+\lambda}_{2,\infty}(\mathbb{R}^{n})$ for arbitrary $s,\lambda\in\mathbb{R}$. The first isomorphism is evident; the second is proved, e.g., in \cite[Theorem 2.3.4]{Triebel95}.
\end{proof}

\begin{remark}\label{rem8.3}
It is well known \cite[Theorem 2.3.2(c)]{Triebel95} that
\begin{equation*}
B^{s}_{2,\infty}(\mathbb{R}^n)\subset H^{s-}(\mathbb{R}^n):=\bigcap_{r<s}H^{r}(\mathbb{R}^n).
\end{equation*}
Proposition \ref{8.1} can be seen as a refinement of this result with the help of the extended Sobolev scale. Thus, e.g., the function $\alpha(t):=t^{s}(1+\log t)^{-\varepsilon-1/2}$ of $t\geq1$ belongs to $\mathrm{OR}$ and satisfies \eqref{int-cond-8.1} for every $\varepsilon>0$, with the space $H^{\alpha}(\mathbb{R}^n)$ being narrower than $H^{s-}(\mathbb{R}^n)$.
\end{remark}

As an immediate consequence of the embedding \eqref{embed} and results in  Section~\ref{sec4}, we obtain \textit{a priori} estimates for solutions to elliptic problems with boundary data in $B_{2,\infty}^s(\Gamma)$. For simplicity of presentation, we discuss a special situation, the formulation in more general settings being obvious. Let $A=A(x,D)=\sum_{|\mu|\le 2} a_\mu(x)D^\mu$ be a properly elliptic second-order PDO on $\overline{\Omega}$, with all $ a_\mu\in C^{\infty}(\overline{\Omega})$. We consider the Dirichlet boundary-value problem
\begin{equation}\label{8-2}
Au=f\;\;\text{ in }\Omega,\quad\gamma_{0}u=g\quad\text{ on }\Gamma,
\end{equation}
where $\gamma_0 u := u\!\upharpoonright\!\Gamma$ denotes the trace of $u$ on the boundary. This is a simple but important example of a regular elliptic problem in~$\Omega$.

Let $\mathrm{OR}_{0}$ denote the set of all $\alpha\in\mathrm{OR}$ such that $\sigma_{0}(\alpha)=\sigma_{1}(\alpha)=0$. In view of Remark~\ref{rem8.3}(a), we restrict ourselves to the case where $\alpha(t)\equiv t^{s}\alpha_{0}(t)$ for some $s\in\mathbb{R}$ and $\alpha_{0}\in\mathrm{OR}_{0}$.

\begin{theorem}\label{8.4}
Assume that a distribution $u\in\mathcal{S}'(\Omega)$ is a generalized solution to the boundary-value problem \eqref{8-2} whose right-hand sides satisfy the conditions $f\in H^\lambda(\Omega)$ and $g\in B_{2,\infty}^s(\Gamma)$ for some numbers $\lambda>-\frac12$ and $s<0$. Then, for every function parameter $\alpha(t)\equiv t^{s+1/2}\alpha_{0}(t)$ such that $\alpha_{0}\in\mathrm{OR}_{0}$ and
\begin{equation}\label{8-3}
\int_1^\infty\alpha_{0}^2(t)\,\frac{dt}{t}<\infty,
\end{equation}
we have $u\in H^\alpha(\Omega)$ and
\begin{equation*}
\|u\|_{\alpha,\Omega}\leq c\,\bigl(\|f\|_{\lambda,\Omega}+ \|g\|_{s,\infty,\Gamma}+\|u\|_{\alpha\rho^{-1},\Omega}\bigr).
\end{equation*}
Here, $\|\cdot\|_{s,\infty,\Gamma}$ denotes the norm in $B^{s}_{2,\infty}(\Gamma)$, and the number $c>0$ does not depend on $u$, $f$, and $g$.
\end{theorem}

\begin{proof}
According to condition~\eqref{8-3} and Proposition~\ref{8.1}, we have the continuous embedding $B^{s}_{2,\infty}(\mathbb{R}^{n-1})\hookrightarrow H^{\alpha\varrho^{-1/2}}(\mathbb{R}^{n-1})$. With the help of local charts on $\Gamma$, we immediately obtain the continuous embedding $B^{s}_{2,\infty}(\Gamma)\hookrightarrow H^{\alpha\varrho^{-1/2}}(\Gamma)$. Now the statement follows directly from Theorems \ref{th4.6} and \ref{th4.12}, in which $\varphi(t)\equiv t^{-2}\alpha(t)\equiv t^{s-3/2}\alpha_{0}(t)$, $\sigma_{0}(\varphi)=\sigma_{1}(\varphi)=s-3/2<-3/2$, and $\eta(t)\equiv t^{\lambda}$.
\end{proof}

The Nikolskii spaces $B^{s}_{2,\infty}(\mathbb{R}^n)$ and $B^{s}_{2,\infty}(\Gamma)$ of order $s<0$ appear in the theory of white noise. We recall the basic definitions. Let $(\widetilde\Omega,\mathcal F,\mathbb P)$ be a probability space, and let $G\in\{\Gamma,\mathbb{R}^n\}$. Then a (spatial) white noise on $G$ is a random variable $\xi\colon \widetilde\Omega\to\mathcal{D}'(G)$ such that for all test functions $v_1,v_2\in\mathcal{D}(G)$ we have
\begin{equation}\label{8-4}
\mathbb{E}[\xi(v_1)\overline{\xi(v_2)}]=C\,(v_1,v_2)_{G}
\end{equation}
with some constant $C>0$. Here, $\mathcal{D}'(G)$ is the topological space of all distributions on $G$, with $\mathcal{D}(G)$ being $C_{0}^{\infty}(\mathbb{R}^n)$ or $C^{\infty}(\Gamma)$. Besides,
$(\cdot,\cdot)_{G}$ denotes the inner product in $L_{2}(G)$, and $\mathbb{E}$ stands for the expectation with respect to $\mathbb P$. A white noise $\xi$ on $G$ is called Gaussian if the scalar random variables $\{\xi(v):v\in\mathcal{D}(G)\}$ are jointly Gaussian with mean zero and with covariance being given by~\eqref{8-4}.

Recently, the Besov space regularity of white noise was studied, e.g., in \cite{FageotFallahUnser17,Veraar11}. For a Gaussian white noise $\xi\colon\widetilde\Omega\to\mathcal{D}'(\mathbb R^n)$, it was shown in \cite[Corollary~3]{FageotFallahUnser17} that $\mathbb P$-almost surely $\xi$ locally belongs to the Besov space $B_{2,r}^s(\mathbb R^n)$ for all $r\in[1,\infty]$ and $s<-n/2$. In \cite{Veraar11}, white noise on the $n$-dimensional torus was studied. It was shown in \cite[Theorem~3.4]{Veraar11} that for
a Gaussian white noise $\xi\colon\widetilde\Omega\to\mathcal{D}'(\mathbb{T}^n)$ we have
$\mathbb P (\xi\in B_{2,\infty}^{-n/2}(\mathbb T^n))=1$.
Here, the upper index is sharp in the sense that for all $s>-n/2$ we have $\mathbb P(\xi\in B_{2,\infty}^s(\mathbb T^n))=0$. Based on these results, one might conjecture that for every Gaussian white noise $\xi$ on an $n$-dimensional closed manifold $M$, we have $\mathbb P(\xi\in B_{2,\infty}^{-n/2}(M))=1$, but this seems to be an open question.

Combining the above regularity of Gaussian white noise with Theorem~\ref{8.4}, we obtain \textit{a priori} estimates for solutions to elliptic problems with boundary noise. As a simple example, we state the result for the Dirichlet Laplacian.

\begin{corollary}\label{8.5}
Consider the boundary-value problem
\begin{equation}\label{8-5}
\Delta u=f\quad\mbox{in}\;\,\Omega,\qquad\gamma_0u=\xi\quad\mbox{on}\;\,\Gamma.
\end{equation}
Here, $\Omega:=\{x\in\mathbb R^2: |x|<1\}$, whereas $\xi$ is a Gaussian white noise on $\Gamma$. Let $f\in H^\lambda(\Omega)$ for some number  $\lambda>-1/2$. Then, for $\mathbb P$-almost all $\omega\in\widetilde \Omega$, there exists a unique pathwise solution $u(\omega,\cdot)$ of \eqref{8-5}, which belongs to $H^\alpha(\Omega)$ for every $\alpha\in \mathrm{OR}_{0}$ subject to \eqref{8-3}. Moreover, for such $\alpha$, the estimate
\begin{equation*}
\|u(\omega,\cdot)\|_{\alpha,\Omega}\leq c_{\alpha}\big(\|f\|_{\lambda,\Omega}+ \|\xi(\omega)\|_{-1/2,\infty,\Gamma}\big)
\end{equation*}
holds $\mathbb P$-almost surely with a number $c_\alpha>0$ that does not depend on $f$, $\xi$, and $\omega$ (but may depend on $\alpha$).
\end{corollary}

\begin{proof}
This is an immediate consequence of Theorem~\ref{8.4} and the fact that $\xi\in B_{2,\infty}^{-1/2}(\Gamma)$ holds $\mathbb P$-almost surely. Note that the unique solvability holds as  $\dim N = \dim N^+ = \{0\}$ for the regular elliptic problem \eqref{8-5}.
\end{proof}

\begin{remark}\label{8.6}
In the last corollary, we have shown that $u\in H^\alpha(\Omega)$ by the embedding result from Proposition~\ref{8.1} and the Nikolskii regularity of white noise. It would be interesting to analyse the regularity of Gaussian white noise (or, more generally,  L\'{e}vy white noise) with respect to the extended Sobolev scale. In particular, this  would allow a direct application of Theorems \ref{th4.6} and \ref{th4.12} for boundary noise.
\end{remark}

\end{document}